\documentclass[]{siamonline220329}

\usepackage{epstopdf}
\usepackage{algorithmic}
\ifpdf
  \DeclareGraphicsExtensions{.eps,.pdf,.png,.jpg}
\else
  \DeclareGraphicsExtensions{.eps}
\fi

\newsiamthm{assumption}{Assumption}

\headers{Sparse Bayesian Inference with Regularized Gaussian Distributions}{J.M. Everink, Y. Dong, M.S. Andersen}

\title{Sparse Bayesian Inference with Regularized Gaussian Distributions\thanks{%Submitted to the editors February 7, 2023.
\funding{This work was supported by the Villum Foundation (grant no.\ 25893 and VIL50096) and the Novo Nordisk Foundation (grant no.\ NNF20OC0061894).}}}

\author{Jasper M. Everink\thanks{Department of Applied Mathematics and Computer Science, Technical University of Denmark. Richard Petersens
Plads, Building 324, DK-2800 Kgs. Lyngby, Denmark. (\href{mailto:jmev@dtu.dk}{jmev@dtu.dk}, \href{mailto:yido@dtu.dk}{yido@dtu.dk}, \href{mailto:mskan@dtu.dk}{mskan@dtu.dk})}\and Yiqiu Dong\footnotemark[2] \and Martin S. Andersen\footnotemark[2]}

\usepackage{float}
\usepackage{bm}

\usepackage{amsmath,amssymb,amsopn,amsfonts, mathabx}

\usepackage{caption}
\usepackage{subcaption}
\usepackage{newfloat}

\DeclareFloatingEnvironment[
    fileext=los,
    listname={Algorithm},
    name=Algorithm,
    within=section,
]{algo}
\usepackage{varwidth}

\usepackage{amsopn}

\DeclareMathOperator{\support}{supp}
\DeclareMathOperator{\closure}{cl}
\DeclareMathOperator{\dom}{dom}
\DeclareMathOperator{\range}{range}

\DeclareMathOperator{\bd}{bd}
\DeclareMathOperator{\relint}{relint}

\DeclareMathOperator*{\argmin}{argmin}
\DeclareMathOperator*{\argmax}{argmax}
\DeclareMathOperator{\Rank}{rank}

\DeclareMathOperator{\conv}{conv}
\DeclareMathOperator{\supp}{supp}
\DeclareMathOperator{\cl}{cl}

\DeclareMathOperator{\prox}{prox}
\newcommand{\nullspace}{{\text{Null}}}
\newcommand{\reals}{\mathbb{R}}
\newcommand{\epi}{\text{epi}}
\newcommand{\symm}{\mathbb{S}}

\newcommand{\prob}{\mathbb{P}}

\renewcommand{\vec}[1]{\bm{#1}}
\newcommand{\set}[1]{\mathbf{#1}}

\usepackage{tikz}
\usetikzlibrary{patterns}
\usepackage{graphicx}
\usepackage{hyperref}

\makeatletter
\newcommand*{\addFileDependency}[1]{% argument=file name and extension
  \typeout{(#1)}% latexmk will find this if $recorder=0 (however, in that case, it will ignore #1 if it is a .aux or .pdf file etc and it exists! if it doesn't exist, it will appear in the list of dependents regardless)
  \@addtofilelist{#1}% if you want it to appear in \listfiles, not really necessary and latexmk doesn't use this
  \IfFileExists{#1}{}{\typeout{No file #1.}}% latexmk will find this message if #1 doesn't exist (yet)
}
\makeatother

\ifpdf
\hypersetup{
  pdftitle={Sparse Bayesian Inference with Regularized Gaussian Distributions},
  pdfauthor={J.M. Everink, Y. Dong, M. S. Andersen}
}
\fi

\begin{document}

\maketitle

\begin{abstract}
Regularization is a common tool in variational inverse problems to impose assumptions on the parameters of the problem. One such assumption is sparsity, which is commonly promoted using lasso and total variation-like regularization. Although the solutions to many such regularized inverse problems can be considered as points of maximum probability of well-chosen posterior distributions, samples from these distributions are generally not sparse. In this paper, we present a framework for implicitly defining a probability distribution that combines the effects of sparsity imposing regularization with Gaussian distributions. Unlike continuous distributions, these implicit distributions can assign positive probability to sparse vectors. We study these regularized distributions for various regularization functions including total variation regularization and piecewise linear convex functions. We apply the developed theory to uncertainty quantification for Bayesian linear inverse problems and derive a Gibbs sampler for a Bayesian hierarchical model. To illustrate the difference between our sparsity-inducing framework and continuous distributions, we apply our framework to small-scale deblurring and computed tomography examples.
\end{abstract}



\section{Introduction}
A common method for reconstructing a signal $\vec{x} \in \reals^n$ from noisy measurements $\vec{b} = A\vec{x} + \vec{e}$, where $A\in \reals^{m\times n}$ is a linear forward operator and $\vec{e}\in \reals^m$ is noise, is to solve a regularized linear least squares problem of the form
\begin{equation}\label{eq:regularized_least_squares} \argmin_{\vec{x}\in\reals^n}\left\{\frac{\lambda}{2} \|A\vec{x} - \vec{b}\|^2_2 + f(\vec{x})\right\},
\end{equation}
where the function $f$ is chosen to make the optimization problem well-posed and to improve the reconstruction by penalizing unwanted behavior. A common choice is to promote sparsity in the reconstruction by choosing $l_1$-based regularization functions of the form $f(\vec{x}) = \gamma \|D\vec{x}\|_1$, where $\gamma > 0$ is the strength of the regularization and $D$ is a linear operator representing the basis in which the reconstruction is to be sparse. Such regularization functions have been studied in areas like compressed sensing with the basis pursuit algorithm \cite{foucart2013compressive} and in image processing with total variation regularization \cite{hansen2010discrete}.

Reconstructing the signal using \eqref{eq:regularized_least_squares} can be motivated by its solution being the maximum a posteriori (MAP) estimate of a suitable posterior distribution, i.e.,
\begin{equation}\label{eq:MAP}
    \argmin_{\vec{x}\in\reals^n}\left\{\frac{\lambda}{2} \|A\vec{x} - \vec{b}\|^2_2 + f(\vec{x})\right\} = \argmax_{\vec{x}\in\reals^n} \pi(\vec{x}\,|\,\vec{b}),
\end{equation}
where,
\begin{equation}\label{eq:posterior_general}
    \pi(\vec{x}\,|\,\vec{b}) \,\propto\, \pi(\vec{b}\,|\,\vec{x}) \pi(\vec{x}) \, \propto \, \exp\left(-\frac{\lambda}{2} \|A\vec{x} - \vec{b}\|^2_2 - f(\vec{x})\right).
\end{equation}
The likelihood function $\pi(\vec{b}\,|\,\vec{x})\, \propto\, \exp\left(-\frac{\lambda}{2} \|A\vec{x} - \vec{b}\|^2_2\right)$ is obtained from assuming that the error $\vec{e} \sim \mathcal{N}(\vec{0}, \lambda^{-1} I)$, i.e., normally distributed with mean zero and covariance $\lambda^{-1} I$, whilst the prior $\pi(\vec{x})\, \propto\, \exp(-f(\vec{x}))$ describes any further assumptions we make on the reconstruction.

Although the MAP estimate \eqref{eq:MAP} can have guaranteed sparsity for suitable regularization functions $f$, the corresponding posterior distribution \eqref{eq:posterior_general} assigns zero probability to these sparse solutions. This is because sparsity is represented by low-dimensional subspaces which are always assigned zero probability by continuous probability distributions like \eqref{eq:posterior_general}.

A method for creating probability distributions that assign positive probability to low-dimensional subspaces is by using variable dimension models, see \cite{robert1999monte}. In these methods, a variety of models $M_i$ are constructed, each having their own probability distribution $\pi(\vec{x}\,|\,M_i)$, possibly supported on a low-dimensional subspace. These model distributions are then combined with a model prior distribution $\pi(M_i)$. For sparsity, each model distribution would be supported on a sparse subspace and the model prior distribution corresponds to assigning a prior distribution on the sparsity of the reconstruction.

Sampling from these variable dimension models is quite challenging. A common method for sampling such distributions are Reversible Jump Markov Chain Monte Carlo (RJMCMC) method \cite{green1995reversible}. This method requires a lot of tuning and easily fails for large dimensional problems. An alternative method is the Shrinkage-Thresholding Metropolis adjusted Langevin
algorithm (STMALA) \cite{schreck2015shrinkage} which provides an efficient proposal for an RJMCMC-like algorithm based on MALA. Although this method works better on larger scale problems, it still requires difficult parameter tuning and can easily result in highly correlated samples.

In \cite{everink2022bayesian}, we showed that for a closed and convex set $\set{C} \subset \reals^n$, the distribution of the constrained linear least squares problem \begin{equation}\label{eq:pertubed_constrained_linear_least_squares}
\argmin_{\vec{x}\in \set{C}}\left\{\frac{\lambda}{2} \|A\vec{x} - \hat{\vec{b}}\|^2_2 + \frac{\delta}{2} \|L\vec{x} - \hat{\vec{c}}\|^2_2\right\},
\end{equation}
with $\nullspace(A) \cap \nullspace(L) = \{\vec{0}\}$, $\hat{\vec{b}} \sim \mathcal{N}(\vec{b}, \lambda^{-1} I)$ and $\hat{\vec{c}} \sim \mathcal{N}(\vec{c}, \delta^{-1} I)$, is identical to sampling from the Gaussian posterior
\begin{equation}\label{eq:posterior_gauss}
    \pi(\vec{x}\,|\,\vec{b}) \,\propto\, \pi(\vec{b}\,|\,\vec{x}) \pi(\vec{x}) \, \propto \, \exp\left(-\frac{\lambda}{2} \|A\vec{x} - \vec{b}\|^2_2 - \frac{\delta}{2} \|L\vec{x} - \vec{c}\|^2_2\right),
\end{equation}
and projecting the sample onto the constraints set $\set{C}$ with respect to the norm $\|\cdot\|_{\lambda A^TA + \delta L^TL}$. Furthermore, if the constraint set $\set{C}$ is polyhedral, then the resulting projected Gaussian assigns positive probability to the various low-dimensional faces of $\set{C}$. The structure of this projected Gaussian distribution is similar to those of variable dimension models and results in independent samples, however, the corresponding probability distribution is implicit and for each sample, an instance of optimization problem \eqref{eq:pertubed_constrained_linear_least_squares} has to be solved accurately.

This method is similar to other optimization-based sampling methods like Perturbation-Optimization \cite{orieux2012sampling} for sampling from Gaussian distributions and Randomize-Then-Optimize (RTO) \cite{wang2017bayesian} for non-linear inverse problems. However, these methods are restricted to sampling from explicitly described continuous probability distributions, whilst our method results in a probability distribution for which we generally do not have an explicit expression.

In this paper, we generalize most of the results in \cite{everink2022bayesian} about \eqref{eq:pertubed_constrained_linear_least_squares}, by studying the optimization problem
\begin{equation}\label{eq:perturbed_regularized_linear_least_squares}\argmin_{\vec{x}\in \reals^n}\left\{\frac{\lambda}{2} \|A\vec{x} - \hat{\vec{b}}\|^2_2 + \frac{\delta}{2} \|L\vec{x} - \hat{\vec{c}}\|^2_2 + f(\vec{x})\right\},
\end{equation}
where $f$ is a sparsity promoting regularization. The case $f(\vec{x}) = \|D\vec{x}\|_1$, where $D$ is a diagonal matrix, has been studied before in \cite{ewald2020distribution}. We mainly consider the class of convex piecewise linear functions $f$. This class includes $l_1$-norm based regularization functions and the polyhedral constraints considered in \cite{everink2022bayesian}. We derive a characterization of the distribution obtained through the regularized and randomized linear least squares problem \eqref{eq:perturbed_regularized_linear_least_squares}. The properties we derive are applied to Bayesian linear inverse problems and hierarchical models and tested using some numerical experiments.

This paper is organized as follows. In Section \ref{sec:regularized_distributions}, we transform problem \eqref{eq:perturbed_regularized_linear_least_squares} into a general framework using proximal operators and study the properties of the resulting regularized Gaussian distributions. We will apply this theory in Section \ref{sec:inverse_problems} to Bayesian linear inverse problems and derive an algorithm for Bayesian hierarchical models in Section \ref{sec:bayes_hierarch}. Finally, numerical experiments based on these models are covered in Section \ref{sec:numerical_examples}.

\section{Regularized Multivariate Distributions}\label{sec:regularized_distributions}

We discuss a theory for combining the effects of regularization to probability distributions. In Subsection \ref{subsec:general_theory}, the definition of a regularized Gaussian is presented together with some general theory about regularized distributions. In Subsection \ref{subsec:sparsity}, the focus will be on regularization functions that introduce sparsity when combined with Gaussian distributions.

\subsection{General theory}\label{subsec:general_theory}

If $A \in \reals^{m\times n}$ has rank $n$, then the regularized least squares problem with random data $\hat{\vec{b}} \sim \mathcal{N}(\vec{b}, \lambda^{-1} I)$ can be transformed as follows,
\begin{align}
&\argmin_{\vec{x}\in \reals^n}\left\{\frac{\lambda}{2} \|A\vec{x} - \hat{\vec{b}}\|^2_2 + f(\vec{x})\right\} \nonumber\\
= &\argmin_{\vec{x}\in \reals^n}\left\{\frac12 \|\vec{x} - (A^TA)^{-1}(A^T\hat{\vec{b}})\|^2_{\lambda A^TA} + f(\vec{x})\right\} \nonumber\\
= &\argmin_{\vec{x}\in \reals^n}\left\{\frac12 \|\vec{x} - \vec{x}^\star\|^2_{\Sigma^{-1}} + f(\vec{x})\right\}, \label{eq:proximal_derivation}
\end{align}
where $\vec{x}^\star = (A^TA)^{-1}(A^T\hat{\vec{b}})$ has a Gaussian distribution with density
\begin{equation}
    \pi_{\vec{x}^\star}(\vec{x})\,\propto\, \exp\left(-\frac{\lambda}{2} \|A\vec{x} - \vec{b}\|^2_2\right),
\end{equation}
and covariance $\Sigma = (\lambda A^TA)^{-1}$. Due to this transformation, we can study the general setting of optimization problem \eqref{eq:proximal_derivation}, where $\vec{x}^\star$ follows any Gaussian distribution. This leads to the following assumption and definition of the object that will be the focus of study in this section.

\begin{assumption}\label{as:gaussian}
Let $\vec{x}^\star$ be a Gaussian random vector satisfying $\vec{x}^\star \sim \mathcal{N}(\vec{\mu}, \Sigma)$ with mean $\vec{\mu} \in \reals^n$ and covariance matrix $\Sigma \in \symm_{++}^n$, where $\symm_{++}^n$ denotes the set of symmetric positive definite real matrices of order $n$.
\end{assumption}

\begin{definition}\label{def:regularized_Gaussian}
Under Assumption \ref{as:gaussian}. If $f : \reals^n \rightarrow \reals \cup \{\infty\}$ is a proper lower semi-continuous convex function, then the Gaussian $\mathcal{N}(\vec{\mu}, \Sigma)$ regularized by $f$ is defined by 
\begin{equation}\label{eq:regularized_Gaussian}
    \prox_{f}^{\Sigma^{-1}}(\vec{x}^\star) := \argmin_{\vec{z}\in\reals^n}\left\{\frac12 \|\vec{z} - \vec{x}^\star\|^2_{\Sigma^{-1}} + f(\vec{z})\right\},
\end{equation}
where $\prox_{f}^{\Sigma^{-1}}$ is the proximal operator with respect to the norm $\|\cdot\|_{\Sigma^{-1}}$.
\end{definition}

Due to its relation to regularized linear least squares, the focus of this work is on regularized Gaussian distributions, but the proximal operator can be applied to any probability distribution. If $\pi$ is any probability measure on $\reals^n$, then by continuity of the proximal operator $\prox_{f}^{\Sigma^{-1}}$, the pushforward measure $\pi \circ \left[\prox_{f}^{\Sigma^{-1}}\right]^{-1}$ is a well-defined probability measure on $\reals^n$. Furthermore, some of the results presented in this work do not rely on $\pi$ being Gaussian and will therefore easily generalize to other distributions.

For most proper lower semi-continuous convex functions $f$, we cannot compute an explicit expression for $\prox_f^{\Sigma^{-1}}$, however, for some simple examples an explicit expression does exist. Let $f(\vec{z}) = \frac12\|D\vec{z}- \vec{d}\|_2^2$ for any $D \in \reals^{k\times n}$ and $\vec{d} \in \reals^{k}$, i.e., generalized Tikhonov regularization. For such regularization, the regularized Gaussian is
\begin{equation*}
    \prox_{f}^{\Sigma^{-1}}(\vec{x}^\star) = (\Sigma^{-1} + D^TD)^{-1}(\Sigma^{-1}\vec{x}^\star + D^T\vec{d}),
\end{equation*}
which is again a Gaussian distribution.  However, in general, the regularized Gaussian is not a Gaussian anymore. For a more complicated example of an explicit description of a regularized Gaussian, see \cite[Lemma II.2]{schreck2015shrinkage}. Let $f(\vec{x}) = \chi_{\set{C}}(\vec{x})$ be the characteristic function of a closed and convex set $\set{C} \subset \reals^n$, i.e., $f(\vec{x})$ is zero on $\set{C}$ and infinite otherwise. Then $\prox_{f}^{\Sigma^{-1}}(\vec{x}^\star) = \Pi^{\Sigma^{-1}}_{\set{C}}(\vec{x}^\star)$, where $\Pi^{\Sigma^{-1}}_{\set{C}}(\vec{x})$ denotes the oblique projection of $\vec{x}$ onto $\set{C}$. In this case, $\prox_{f}^{\Sigma^{-1}}(\vec{x}^\star)$ has the same density as $\vec{x}^\star$ on the interior of $\set{C}$, but all the mass outside of $\set{C}$ gets transported to the boundary of $\set{C}$. We have studied this projection setting specifically in \cite{everink2022bayesian}. 

In general, the following proposition describes the support of the regularized Gaussian.
\begin{proposition}\label{prop:regularized_support}
Under Assumption \ref{as:gaussian}. If $f : \reals^n \rightarrow \reals \cup \{\infty\}$ is a proper lower semi-continuous convex function, then,
\begin{equation*}
    \supp(\prox_{f}^{\Sigma^{-1}}(\vec{x}^\star)) = \cl(\dom(f)),
\end{equation*}
where $\cl(\set{C})$ denotes the closure of $\set{C}$ and $\supp(\vec{z}^\star)$ denotes the support of the random vector $\vec{z}^\star$, i.e, the set of points for which any open set around that point is assigned positive probability by $\vec{z}^\star$.
\end{proposition}
\begin{proof}
Use that $\range(\prox_{f}^{\Sigma^{-1}}) = \dom(f)$ in Lemma \ref{lemma:pushforward_support}.
\end{proof}

Note that for the case of $f(\vec{x}) = \chi_{\set{C}}(\vec{x})$ the projection $\Pi^{\Sigma^{-1}}_{\set{C}}$ maps all samples outside the constraint set $\set{C}$ to the boundary $\bd(\set{C})$, which is a set with Lebesgue measure zero. Furthermore, it can be shown \cite[Lemma 2.4]{everink2022bayesian} that $\Pi^{\Sigma^{-1}}_{\set{C}}(\vec{x}^\star)$ assigns positive probability to $\bd(\set{C})$ if $\set{C} \neq \reals^n$, thus $\Pi^{\Sigma^{-1}}_{\set{C}}(\vec{x}^\star)$ does not have a density with respect to the Lebesgue measure and is therefore not a continuous random variable. The following proposition gives a condition on the smoothness of $f$ such that the regularized Gaussian $\prox_{f}^{\Sigma^{-1}}(\vec{x}^\star)$ is a continuous random variable.

\begin{proposition}\label{prop:lipschitz_density}
Under Assumption \ref{as:gaussian}, let $f : \reals^n \rightarrow \reals \cup \{\infty\}$ with $\dom(f)$ open. If $f$ is an everywhere differentiable convex function on $\dom(f)$ with locally Lipschitz continuous gradient, then  $\prox_{f}^{\Sigma^{-1}}(\vec{x}^\star)$ is a continuous distribution on $\dom(f)$.
\end{proposition}
\begin{proof}
Note that the optimality condition associated with $\vec{z} = \prox_{f}^{\Sigma^{-1}}(\vec{x})$ equals
\begin{equation}\label{eq:opt_condition}
    \vec{x} \in \vec{z} + \Sigma \partial f(\vec{z}),
\end{equation}
and hence the inverse image of the proximal operator can be expressed as the set-valued map
\begin{equation}\label{eq:inverse_image_prox}
    \left[\prox_{f}^{\Sigma^{-1}}\right]^{-1}(\vec{z}) = \vec{z} + \Sigma\partial f(\vec{z}).
\end{equation}
Thus, by assumption, $\left[\prox_{f}^{\Sigma^{-1}}\right]^{-1}(\vec{z}) = \vec{z} + \Sigma\nabla f(\vec{z})$ is a locally Lipschitz continuous single valued map, hence the claim follows from Corollary \ref{lemma:locally_Lipschitz_density}.
\end{proof}

Proposition \ref{prop:lipschitz_density} shows that if the regularization function $f$ is sufficiently smooth, then the regularized Gaussian has a continuous distribution. It does not consider the case of a regularization function that is differentiable but without a locally Lipschitz continuous gradient\footnote{Proposition \ref{prop:lipschitz_density} can be generalized to differentiable convex functions $f$ such that $\vec{z} \mapsto \vec{z} + \Sigma\nabla f(\vec{z})$ preserves Lebesgue measurability.}. For many non-differentiable functions, we can show the converse. The following theorem gives a condition on functions such that their corresponding regularized distribution is not continuous.

\begin{theorem}\label{thm:low_dimensional}
Under Assumption \ref{as:gaussian}, let $f: \reals^n \rightarrow \reals \cup \{\infty\}$ be a proper, lower semi-continuous convex function.
Furthermore, let $\set{U} \subset \reals^k$ be compact with $s : \set{U} \rightarrow s(\set{U}) \subseteq \reals^n$ a homeomorphism.
Furthermore, let $\set{V} \subset \reals^{n-k}$ be compact with $d : \set{U} \times \set{V} \rightarrow d(\set{U}, \set{V}) \subseteq \reals^n$ continuous and invertible in $\vec{v}$ such that for any $\vec{u} \in \set{U}$,
\begin{equation*}
    d(\vec{u}, \set{V}) \subseteq \partial f(s(\vec{u})).
\end{equation*}
Then,
\begin{equation*}
    \prob \left(\prox_f^{\Sigma^{-1}}(\vec{x}^{\star}) \in s(\set{U})\right) > 0.
\end{equation*}
\end{theorem}

\begin{proof}
Note that the event $\{\prox_f^{\Sigma^{-1}}(\vec{x}^{\star}) \in s(\set{U})\}$ is equivalent to the event $\{\vec{x}^{\star} \in (\text{Id} + \Sigma \partial f)(s(\set{U}))\}$. Because $\support(\vec{x}^\star) = \reals^n$, it is enough to show that the set $(\text{Id} + \Sigma \partial f)(s(\set{U}))$ contains an open set.

By assumption, the set $(\text{Id} + \Sigma \partial f)(s(\set{U}))$ contains the image of the map
\begin{equation*}
    h: \set{U} \times \set{V} \rightarrow \reals^n, \quad \text{defined by }\quad h(\vec{u}, \vec{v}) = s(\vec{u}) + \Sigma d(\vec{u}, \vec{v}).
\end{equation*}

If $\vec{u}, \vec{u}' \in \set{U}$ and $\vec{v}, \vec{v}' \in \set{V}$ such that $h(\vec{u}, \vec{v}) = h(\vec{u}', \vec{v}')$. Then
\begin{equation*}
    s(\vec{u}) = \prox_f^{\Sigma}(h(\vec{u}, \vec{v})) = \prox_f^{\Sigma}(h(\vec{u}', \vec{v}')) = s(\vec{u}'),
\end{equation*}
hence, because $s$ is a homeomorphism, $\vec{u} = \vec{u}'$. Furthermore, because $s(\vec{u}) + \Sigma d(\vec{u}, \vec{v}') = s(\vec{u}) + \Sigma d(\vec{u}, \vec{v}')$ and $d$ is invertible in $\vec{v}$, we get $\vec{v} = \vec{v}'$. Thus, $h$ is a homeomorphism between the $n$-dimensional set $\set{U} \times \set{V}$ and a subset of $(\text{Id} + \Sigma \partial f)(s(\set{U}))$. Therefore, we can conclude that $(\text{Id} + \Sigma \partial f)(s(\set{U}))$ contains an open set.
\end{proof}

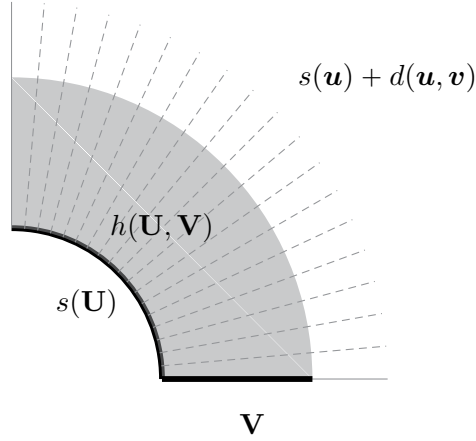
\begin{figure}
    \centering
    \begin{tikzpicture}[scale=2]
        % Main shape
        \draw[line width = 0.2em] (1,0) arc (0:90:1);
        %\draw[line width = 0.2em] (1.5,0) arc (0:90:1.5);
        %\draw[line width = 0.2em] (0,1) -- (0,1.5);
        %\draw[line width = 0.2em] (1,0) -- (1.5,0);
        %\draw[gray] (0,0) -- (0,1);
        %\draw[gray] (0,0) -- (1,0);

        \fill[gray, opacity = 0.5] (1,0) arc (0:90:1) -- (0, 2) -- (2, 0) arc (0:90:2) -- (2, 0);
        
        % Normal cone boundaries
        \draw[gray] (1,0) -- (2.5,0);
        \draw[gray] (0,1) -- (0,2.5);
        % Illustrate normal projection
        \foreach \phi in {5,10,...,85}{\draw[gray, densely dashed] ({cos(\phi)},{sin(\phi)}) -- ({2.5*cos(\phi)},{2.5*sin(\phi)});}

        \node[] at (1.6,-0.3) {$\set{V}$};
        \draw[line width = 0.2em] (1,0) -- (2,0);
        
        \node[] at (2.5,2) {$s(\vec{u}) + d(\vec{u}, \vec{v})$};
        \node[] at (0.5,0.5) {$s(\set{U})$};
        \node[] at (1,1) {$h(\set{U}, \set{V})$};
        
    \end{tikzpicture}
    
    \caption{Visualization of Theorem \ref{thm:low_dimensional} for the curved boundary of a quarter disc.}
    \label{fig:projection}
\end{figure}

Theorem \ref{thm:low_dimensional} covers most cases of interest, including most sparsity-imposing regularization functions as will be discussed in Subsection \ref{subsec:sparsity}.
As a separate example, let $\set{C}$ be a quarter disc, i.e., the unit disc intersected with the nonnegative orthant, and consider $f(\vec{x}) = \chi_{\set{C}}(\vec{x})$, i.e., constraining the linear least squares problem to the quarter disc $\set{C}$. The subdifferential of $f$ is the normal cone $ \partial f(\vec{x}) = N_{\set{C}}(\vec{x})$. Consider $\Sigma = I$ for simplicity and let $s(\set{U})$ be the one-dimensional curved part of the boundary as shown in Figure \ref{fig:projection}. The subdifferential of $f$ over $s$ can be parameterized as $d(u, v) = v s(u)$ for $v \geq 0$. All mass from $s(\vec{u}) + \set{N}_{\set{C}}(s(\vec{u}))$, including the subset $h(\set{U}, \set{V})$ obtained by restricting $v$ to $[0, 1]$, gets projected onto $s(\set{U})$, resulting in $\prox_f^{\Sigma^{-1}}(\vec{x}^\star)$ to assign a positive probability to $s(\set{U})$.

\subsection{Sparsity}\label{subsec:sparsity}

Although Theorem \ref{thm:low_dimensional} could apply to any surface, commonly, the focus is put on linear subspaces. Particularly, if we are interested in solutions that lie in a subspace defined by the linear system $D\vec{x} = \vec{d}$, we add as regularization function $\|D\vec{x} - \vec{d}\|$, where the norm is non-differentiable at zero, hence the regularization function is non-differentiable in the specified subspace. The following theorem shows that for these kinds of regularization functions, the regularized Gaussian has positive probability on these subspaces.

\begin{theorem}\label{thm:group_sparsity}
Let $p \geq 1$,  $D_i \in \reals^{c_i \times n}$ and $\vec{d}_i \in \reals^{c_i}$ for $i=1,\dots,k$. Then define a proper continuous convex function $f : \reals^n \rightarrow \reals$ by
\begin{equation}\label{eq:group_sparsity}
    f(\vec{x}):= \sum_{i=1}^{k}\|D_i\vec{x} - \vec{d}_i\|_p.
\end{equation}
For any $\set{I} \subseteq \{1, \dots, k\}$ let $\set{F}_{\set{I}}$ be a nonempty compact subset of $\{\vec{x} \in \reals^n\,|\, D_i\vec{x}-\vec{d}_i = 0, \forall i \in \set{I} \text{ and } D_i\vec{x}-\vec{d}_i \neq 0, \forall i \not\in \set{I}\}$. Then,
\begin{equation*}
    \prob \left(\prox_f^{\Sigma^{-1}}(\vec{x}^{\star}) \in \set{F}_{\set{I}}\right) > 0.
\end{equation*}
\end{theorem}

\begin{proof}
    Assume first that $p > 1$. For $\vec{x} \in \set{F}_{\set{I}}$, the subdifferential of $f$ can be decomposed as follows. The mapping $\vec{x} \mapsto \|D_i\vec{x} - \vec{d}_i\|_p$ is differentiable for any $i \not\in \set{I}$. For $i \in \set{I}$, $\vec{x} \mapsto \|D_i\vec{x} - \vec{d}_i\|_p$ is non-differentiable with subgradient $D_i^T\partial(\|\cdot\|_p)(\vec{0})$, where $\partial(\|\cdot\|_p)(\vec{0})$ is an $c_i$-dimensional unit $l_p$-norm ball around $\vec{0}$. We can therefore split the subgradient of $f$ into a set that is constant over $\set{F}_{\set{I}}$ and a continuous translation, i.e.,
    \begin{equation}\label{eq:group_sparsity_decomposition}
        \partial f(\vec x) =  \set{E} + \nabla g(\vec{x}),
    \end{equation}
    with 
    \begin{align*}
        \set{E} &= \sum_{i \in \set{I}}D_i^T\partial(\|\cdot\|_p)(\vec{0}) \quad \text{and} \\
        g(\vec{x}) &= \sum_{i \not\in \set{I}}  \|D_i\vec{x} - \vec{d}_i\|_p,
    \end{align*}
    where the former sum is the Minkowski sum.
    Note that $\set{E}$ is orthogonal to $\set{F}_{\set{I}}$ and $\dim(\set{E}) + \dim(\set{F}_{\set{I}}) = n$. Thus, the conclusion follows from Theorem \ref{thm:low_dimensional}.
    
    For the case $p = 1$, note that any scalar $a \in \reals$ satisfies $\|a\|_2 = |a|$. Hence, for any $\vec{z} \in \reals^d$ we have $\|\vec{z}\|_1 = \sum_{j = 1}^{d}\|\vec{z}_j\|_2$, where $\vec{z}_j$ denotes the components of $\vec{z}$.
\end{proof}

Examples of functions of the form \eqref{eq:group_sparsity} include regularization functions like (an)isotropic total variation. However, if the differentiable part $g$ in the decomposition \eqref{eq:group_sparsity_decomposition} is zero, a lot more can said. The following assumption restricts regularization functions to those functions for which the different, possibly low-dimensional, subspaces have a constant sub-differential.

\begin{assumption}\label{as:polyhedral_epigraph}
Let $f : \mathbb{R}^n \rightarrow \mathbb{R} \cup \{\infty\}$ be any proper function whose epigraph is a polyhedral set, or equivalently, $f$ is a convex piecewise linear function \cite[Theorem 2.49]{rockafellar2009variational}. Note that any such function can be written as 
\begin{equation*}
f(\vec{x}) = \chi_{\set{C}}(\vec{x}) + \max_{i\in \set{I}} \{\vec{a}_i^T\vec{x} + b_i\},  
\end{equation*} where $\set{C} \subseteq \mathbb{R}^n$ is a polyhedral set, $\set{I}$ is a finite index set and the pairs $(\vec{a}_i, b_i)$ define affine functions. Note that $\set{C}$ is the effective domain of $f$.
\end{assumption}

Examples of function that satisfy Assumption \ref{as:polyhedral_epigraph} include polyhedral constraints like nonnegativity ($f(\vec{z}) = \chi_{\reals_+^n}(\vec{z})$), and $l_1$-norm based regularization like anisotropic total variation ($f(\vec{z}) = \gamma \|D\vec{z}\|_1$) and certain Besov norms \cite{wang2017bayesian} ($f(\vec{z}) = \gamma \|WB\vec{z}\|_1$). Note that isotropic total variation and more generally group sparsity ($f(\vec{z}) = \sum_i \|D_i\vec{z}\|_2$) is not included in this class.

The following definition gives a partition of the domain of the regularization functions, such that each element corresponds to part of a possibly low-dimensional subspace of interest. 

\begin{definition}\label{def:polyhedral_partition}
Under Assumption \ref{as:polyhedral_epigraph}, we define the polyhedral partition $\{\set{F}_j\}_{j\in \set{J}}$ of $\dom(f)$ as follows. Denote by $\hat{\set{F}}_j \subset \mathbb{R}^{n+1}$ the face of the polyhedral epigraph $\epi(f)$ for which $(\vec{x}, t) \in \hat{\set{F}}_j$ implies $f(\vec{x}) = t$. Denote by $\set{F}_i$ the projection $(\vec{x}, t) \mapsto \vec{x}$ applied to $\relint(\hat{\set{F}}_i)$, where $\relint$ denotes the relative interior.

%Under Assumption \ref{as:polyhedral_epigraph}, we define the polyhedral partition of $\dom(f)$ as follows. Denote by $\hat{\set{F}}_j \subset \mathbb{R}^{n+1}$ the relative interiors of the non-interior faces of the epigraph of $f$. Denote by $\set{F}_i$ the projection $(\vec{x}, t) \mapsto \vec{x}$ applied to $\hat{\set{F}}_i$. After identification of equal $\set{F}_j$, the set $\{\set{F}_j\}_{j\in \set{J}}$ is a partition of the domain $\set{F}$.

Equivalently, $\{\set{F}_j\}_{j\in \set{J}}$ is the partition of $\dom(f)$ by equivalence of subdifferentials, i.e., for all $j\in \set{J}$, we have $\vec{x}, \vec{y} \in \set{F}_j$ if and only if $\partial f(\vec{x}) = \partial f(\vec{y})$.
\end{definition}

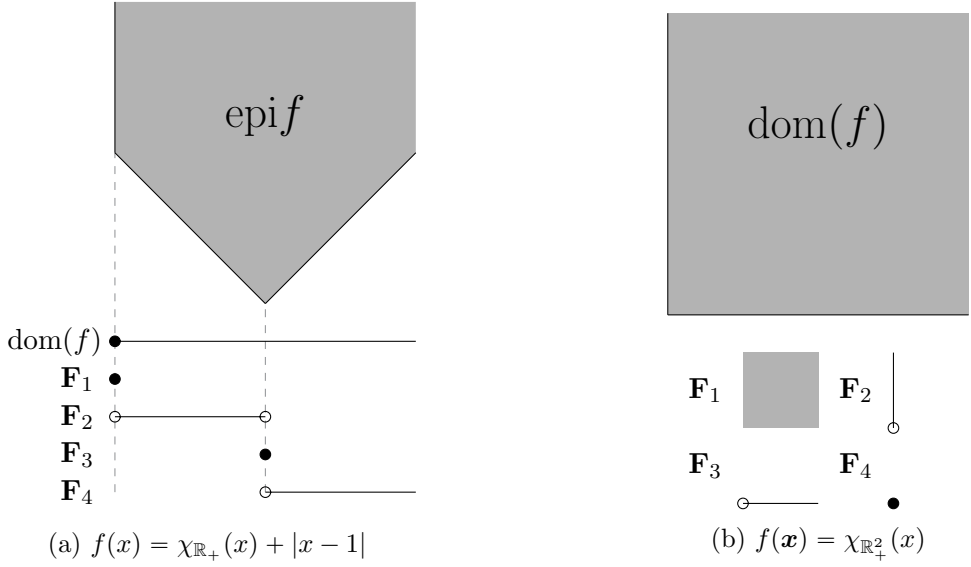
\begin{figure}
    \centering
    \begin{subfigure}[b]{0.48\textwidth}
        \centering
        \begin{tikzpicture}[scale=2]
            \fill[fill = black!30!white] (0,1) -- (1,0) -- (2,1) -- (2,2) -- (0, 2) -- cycle;
        
            \draw[gray, dashed] (0,-1.25) -- (0,1);
            \draw[gray, dashed] (1,-1.25) -- (1,0);
            
            \draw[black] (0,1) -- (0,2);
            \draw[black] (0,1) -- (1,0);
            \draw[black] (1,0) -- (2,1);
            
            \node[] at (1,1.25) {\LARGE $\epi{f}$};
            
            \filldraw (0,-0.25) circle (1pt) -- (2,-0.25);
            \node[] at (-0.4,-0.25) {$\dom (f)$};
            
            \filldraw (0,-0.5) circle (1pt);
            \node[] at (-0.25,-0.5) {$\set{F}_1$};
            
            \draw (0,-0.75)  circle (1pt) -- (1,-0.75)  circle (1pt);
            \node[] at (-0.25,-0.75) {$\set{F}_2$};
            
            \filldraw (1,-1)  circle (1pt);
            \node[] at (-0.25,-1) {$\set{F}_3$};
            
            \draw (1,-1.25)  circle (1pt) -- (2,-1.25);
            \node[] at (-0.25,-1.25) {$\set{F}_4$};
            
        \end{tikzpicture}
        \caption{$f(x) = \chi_{\mathbb{R}_+}(x) + |x-1|$}
        \label{fig:ex_polyhedral_partition_1d}
    \end{subfigure}
    \hfill
    \begin{subfigure}[b]{0.48\textwidth}
        \centering
        \begin{tikzpicture}[scale=2]
        \fill[fill = black!30!white] (0,0) -- (0,2) -- (2,2) -- (2,0) -- cycle;
        \draw[black] (0,0) -- (2,0);
        \draw[black] (0,0) -- (0,2);
    
        \node[] at (1,1.25) {\LARGE $\dom (f)$};
        
        \fill[fill = black!30!white] (0.5+0,-0.75) -- (0.5+0.5,-0.75) -- (0.5+0.5,-0.25) -- (0.5+0, - 0.25) -- cycle;
        \node[] at (0.5+-0.25,-0.5) {$\set{F}_1$};
        
        \draw (0.5+1,-0.75)  circle (1pt) -- (0.5+1,-0.25);
        \node[] at (0.5+0.75,-0.5) {$\set{F}_2$};
                
        \draw (0.5+0,-1.25)  circle (1pt) -- (0.5+0.5,-1.25);
        \node[] at (0.5+-0.25,-1) {$\set{F}_3$};
    
        \filldraw (0.5+1,-1.25)  circle (1pt);
        \node[] at (0.5+0.75,-1) {$\set{F}_4$};

        \end{tikzpicture}
        \caption{$f(\vec{x}) = \chi_{\mathbb{R}_+^2}(x)$}
        \label{fig:ex_polyhedral_partition_2d}
    \end{subfigure}
    \caption{Examples of polyhedral partitions in one and two dimensions.}
    \label{fig:ex_polyhedral_partition}
\end{figure}

Figure \ref{fig:ex_polyhedral_partition} illustrates this polyhedral partition. Note that in some cases, the partition contains many sets of low dimension. We will show that the regularized Gaussian assigns positive probability to each set of the polyhedral partition. Furthermore, the regularized Gaussian can be described by a mixture of possibly low-dimensional densities on each set and these densities are proportional to the original Gaussian up to translation. The following lemma describes an orthogonality property which holds due to the translation component $\nabla g$ in \eqref{eq:group_sparsity_decomposition} being zero. 

\begin{lemma}\label{lemma:constant_subdifferential}
Consider the polyhedral partition $\{\set{F}_j\}_{j\in \set{J}}$ of Definition \ref{def:polyhedral_partition}. The subdifferential $\partial f(\vec{x})$ for $ \vec{x} \in \set{F}_i$ does not depend on $\vec{x}$, hence will be denoted by $\partial f_{\set{F}_j}$. Furthermore, parameterize $\set{F}_j$ as $\vec{x}_j + F_j\vec{u}$ for fixed $\vec{x}_j \in \set{F}_j$, $F_j \in \mathbb{R}^{n\times \dim(\set{F}_j)}$, and local coordinates $\vec{u} \in \set{U} \subseteq \mathbb{R}^{\dim(\set{F}_j)}$. Then, there exists $\vec{d} \in \{\vec{0}, \vec{a}_1,\dots, \vec{a}_{|\set{I}|} \}$ such that $(\partial f_{\set{F}_j} - \vec{d})^TF_j\vec{u} = \{0\}$.
\end{lemma}

\begin{proof}
For each $\set{F}_j$, there exists a largest, non-empty set of active affine functions $\hat{\set{I}} \subseteq \set{I}$, such that $f|_{\set{F}_j}(\vec{x}) = \vec{a}_i^T\vec{x} + b_i$ for all $i \in \hat{\set{I}}$. 
For the first part, note that for $\vec{x} \in \set{F}_j$, $\partial f(\vec{x}) = N_\set{\dom(f)}(\vec{x}) + \conv\{\vec{a}_i\,|\, i \in \hat{\set{I}}\}$, where $\conv\{\set{S}\}$ denotes the convex hull of $\set{S}$. Hence, $\partial f(\vec{x})$ is independent of $\vec{x} \in \set{F}_j$.

For the second part, note that for any $j, k\in\hat{\set{I}}$ and $\vec{u} \in \set{U}$,
\begin{align*}
    \vec{a}_j^T(\vec{x}_j + F_j\vec{u}) + b_j &= \vec{a}_k^T(\vec{x}_j + F_j\vec{u}) + b_k \quad \text{and}\\
    \vec{a}_j^T\vec{x}_j + b_j &= \vec{a}_k^T\vec{x}_j + b_k.
\end{align*}
From these equations we obtain
\begin{equation*}
    0 = (\vec{a}^T_j\vec{x}_j + b_j) - (\vec{a}^T_k\vec{x}_j + b_k) + (\vec{a}_j - \vec{a}_k)^TF_j\vec{u} = (\vec{a}_j - \vec{a}_k)^TF_j\vec{u}.
\end{equation*}
Therefore, for any $\vec{v} \in \conv\{\vec{a}_j\,|\, j \in \hat{\set{I}}\}$ we have $(\vec{v} - \vec{a}_k)^TF_j\vec{u} = 0$.
Furthermore, for any $\vec{n} \in N_\set{\dom(f)}(\vec{x}_j + F_j\vec{u})$ we have $\vec{n}^TF_j\vec{u} = 0$.
Thus, we can conclude that for any $\vec{v} \in \partial f(\vec{x}) = N_{\dom(f)}(\vec{x}) + \conv\{\vec{a}_i\,|\, i \in \hat{\set{I}}\}$, we have $(\vec{v} - \vec{a}_k)^TF_j\vec{u} = 0$.

%If there is no active $\vec{a}_k$, then the zero vector suffices.
\end{proof}

\begin{theorem}\label{thm:gauss_decomposition}
Under Assumptions \ref{as:gaussian} and \ref{as:polyhedral_epigraph}. Consider the polyhedral partition $\{\set{F}_j\}_{j\in \set{J}}$ of Definition \ref{def:polyhedral_partition}. The probability of $\prox_f^{\Sigma^{-1}}(\vec{x}^\star)$ conditioned on $\set{F}_j$ can be described by a density $\pi_{\set{F}_j}(\vec{x})$ on $\set{F}_j$ that is proportional to a Gaussian with covariance $\Sigma$.
Furthermore, we have
\begin{equation*}
    \pi_{\set{F}_i}(\vec{x}) \,\propto\, \exp\left(-\frac12\|\vec{x}-\vec{\mu}\|_{\Sigma^{-1}}^2 - f(\vec{x})\right).
\end{equation*}
\end{theorem}

\begin{proof}
The optimality condition of \eqref{eq:regularized_Gaussian} is $\vec{x}^\star \in \vec{z}^\star + \Sigma \partial f(\vec{z}^\star)$. For $\vec{z}^\star \in \set{F}_j$, this simplifies to $\vec{x}^\star \in \vec{z}^\star + \Sigma \partial f_{\set{F}_j}$. Let $\set{E} \subseteq \set{F}_j$ be measurable, then we can write
\begin{align*}
    \prob\left(\prox_f^{\Sigma^{-1}}(\vec{x}^\star) \in \set{E}\right) &= \int_{\set{E} + \Sigma \partial f_{\set{F}_j}} \pi_{\vec{x}^\star}(\vec{w})\text{d}\vec{w}\\
    &= \int_{\set{E}}\int_{\Sigma \partial f_{\set{F}_j}} \pi_{\vec{x}^\star}(\vec{z} + \vec{v})\text{d}\vec{z}\text{d}\vec{v}
\end{align*}
where $\pi_{\vec{x}^\star}$ is the density of $\vec{x}^\star \sim \mathcal{N}(\vec{\mu}, \Sigma)$. Hence, we can write the probability in terms of a density $\pi_{\set{F}_j}$ like
\begin{equation*}
    \prob\left(\prox_f^{\Sigma^{-1}}(\vec{x}^\star) \in \set{E}\right) = \int_{\set{E}}\pi_{\set{F}_j}(\vec{z})\text{d}\vec{z},
\end{equation*}
with $\pi_{\set{F}_j}(\vec{z}) = \int_{\Sigma \partial f_{\set{F}_j}} \pi_{\vec{x}^\star}(\vec{z} + \vec{w})\text{d}\vec{w}$.

Letting $\vec{z} = \vec{x}_j + F_j\vec{u}$, we can write
\begin{equation*}
\pi_{\set{F}_j}(\vec{u})\ \propto \int_{\partial f_{\set{F}_j}} \pi_{\vec{x}^\star}(\vec{x}_j + F_j\vec{u} + \Sigma \vec{v}) \text{d} \vec{v} = \int_{\partial f_{\set{F}_j} - \vec{d}} \pi_{\vec{x}^\star}(\vec{x}_j + F_j\vec{u} + \Sigma \vec{d}  + \Sigma \vec{v}) \text{d}\vec{v}.
\end{equation*}
The orthogonality of $F_j\vec{u}$ and $\partial f_{\set{F}_j} - \vec{d}$ from Lemma \ref{lemma:constant_subdifferential} implies that for any $\vec{v} \in \partial f_{\set{F}_j} - \vec{d}$,
\begin{align*}
    -2\log\left(\frac{\pi_{\vec{x}^\star}(\vec{x}_j + F_j\vec{u} + \Sigma \vec{d}  + \Sigma \vec{v})}{\pi_{\vec{x}^\star}(\vec{x}_j + F_j\vec{u} + \Sigma \vec{d})}\right) &= 2(\vec{x}_j + F_j\vec{u} + \Sigma \vec{d})^T\vec{v} + \vec{v}^T\Sigma \vec{v} \\
    &= 2(\vec{x}_j + \Sigma \vec{d})^T\vec{v} + \vec{v}^T\Sigma \vec{v},
\end{align*}
hence independent of $\vec{u}$, from which we can be conclude that
\begin{equation*}
\pi_{\set{F}_j}(\vec{u})\ \propto\ \pi_{\vec{x}^\star}(\vec{x}_j + F_j\vec{u} + \Sigma \vec{d}),
\end{equation*}
where $\pi_{\vec{x}^\star}(\cdot + \Sigma \vec{d})$ is the density of a normal distribution with mean $\vec{\mu} - \Sigma \vec{d}$ and covariance $\Sigma$.

For the second part, note that
\begin{equation*}
    \frac12\|\vec{x} - \vec{\mu} + \Sigma \vec{d}\|_{\Sigma^{-1}}^2 = \frac12\|\vec{x} - \vec{\mu}\|_{\Sigma^{-1}}^2 + \vec{d}^T\vec{x} - \vec{d}^T\vec{\mu} + \frac12\|\Sigma\vec{d}\|_{\Sigma^{-1}}^2.
\end{equation*}
Only the first two terms depend on $\vec{x}$. Furthermore, by Lemma \ref{lemma:constant_subdifferential} we have $f|_{\set{F}_i}(\vec{x}) = \vec{d}^T\vec{x} + b_j$, and hence for any $\vec{x} \in \set{F}_i$,
\begin{equation*}
    \frac12\|\vec{x} - \vec{\mu}\|_{\Sigma^{-1}}^2 + \vec{d}^T\vec{x} = \frac12\|\vec{x} - \vec{\mu}\|_{\Sigma^{-1}}^2 + f(\vec{x}).
\end{equation*}
Thus, we can conclude that
\begin{equation*}
    \pi_{\set{F}_j}(\vec{x})\,\propto\, \exp\left(-\frac12\|\vec{x} - \vec{\mu} + \Sigma \vec{d}\|_{\Sigma^{-1}}^2\right) \,\propto\, \exp\left(-\frac12\|\vec{x}-\vec{\mu}\|_{\Sigma^{-1}}^2 - f(\vec{x})\right).
\end{equation*}
\end{proof}

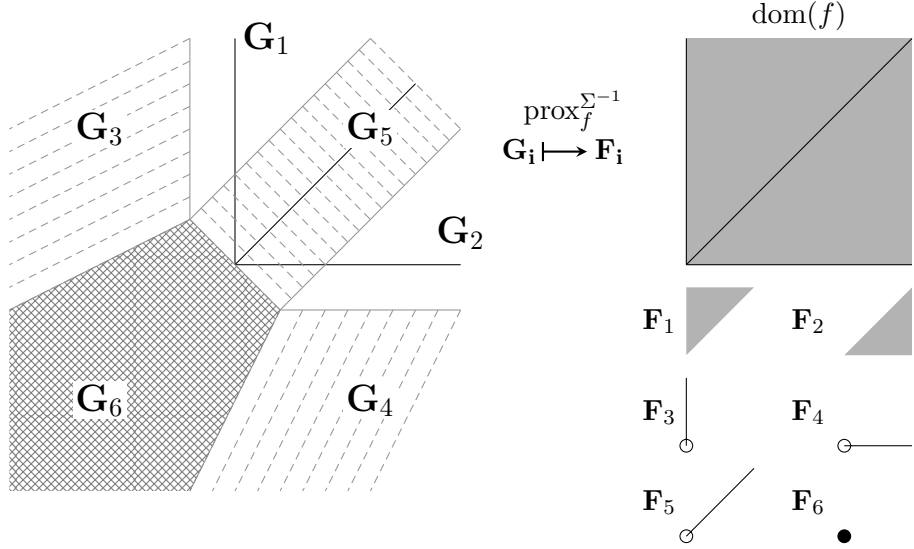
\begin{figure}[H]
    \centering
    \begin{tikzpicture}[scale=0.6]
        % Untransformed
        \draw[black] (0,0) -- (5,0);
        \draw[black] (0,0) -- (0,5);
        \draw[black] (0,0) -- (4,4);
        
        \draw[gray] (-1,1) -- (1,-1);
        \draw[gray] (-1,1) -- (-5,-1);
        \draw[gray] (1,-1) -- (-1,-5);
        \draw[gray] (-1,1) -- (-1,5);
        \draw[gray] (1,-1) -- (5,-1);
        \draw[gray] (-1,1) -- (3,5);
        \draw[gray] (1,-1) -- (5,3);
        
        \foreach \x in {0, 0.5,...,4}{\draw[gray, densely dashed] ({1+\x},-1) -- ({-1 + \x},-5);}
        \foreach \y in {0, 0.5,...,4}{\draw[gray, densely dashed] (-1,{1+\y}) -- (-5,{-1+\y});}
        \foreach \d in {0, 0.25,...,4}{\draw[gray, densely dashed] ({-1+\d},{1+\d}) -- ({1+\d},{-1+\d});}
        
        \fill[pattern=crosshatch, pattern color=gray] (-1,-5) -- (1,-1) -- (-1,1) -- (-5, -1) -- (-5, -5);  
        
        % Offset x + 10
        %\fill[fill = black!30!white] (10,0) -- (15,0) -- (15,5) -- (10,5) -- cycle;
        %\draw[black] (10,0) -- (15,0);
        %\draw[black] (10,0) -- (10,5);
        %\draw[black] (10,0) -- (15,5);
        
        \node[] at (7.5,3.4) {$\prox_f^{\Sigma^{-1}}$};
        \draw [thick, |-stealth](6.8,2.5) -- (7.8,2.5);
        
        \node[] at (6.3,2.5) {$\set{G_i}$};
        \node[] at (8.3,2.5) {$\set{F_i}$};
        
        % Left Es
        \node [fill=white,inner sep=1pt] at (-3,-3) {\Large $\set{G}_6$};
        \node [fill=white,inner sep=1pt] at (-3,3) {\Large $\set{G}_3$};
        \node [fill=white,inner sep=1pt] at (3,-3) {\Large $\set{G}_4$};
        \node [fill=white,inner sep=1pt] at (3,3) {\Large $\set{G}_5$};
        \node [fill=white,inner sep=1pt] at (0.7,5) {\Large $\set{G}_1$};
        \node [fill=white,inner sep=1pt] at (5,0.7) {\Large $\set{G}_2$};
        
        % Right Fs
        % \node [fill=white,inner sep=1pt] at (11.5,5.5) {\Large $\set{F}_1$};
        % \node [fill=white,inner sep=1pt] at (15.5,1.5) {\Large $\set{F}_2$};
        % \node [fill=white,inner sep=1pt] at (15.5,5.5) {\Large $\set{F}_5$};
        % \node [fill=white,inner sep=1pt] at (9.4,3) {\Large $\set{F}_3$};
        % \node [fill=white,inner sep=1pt] at (13,-0.6) {\Large $\set{F}_4$};
        % \node [fill=white,inner sep=1pt] at (9.4,-0.6) {\Large $\set{F}_6$};
        
        \fill[fill = black!30!white] (10,0) -- (15,0) -- (15,5) -- (10,5) -- cycle;
        \draw[black] (10,0) -- (10,5);
        \draw[black] (10,0) -- (15,5);
        \draw[black] (10,0) -- (15,0);
    
        \node[] at (12.5,5.5) {$\dom (f)$};
        
        \fill[fill = black!30!white] (10,-2) -- (11.5,-0.5) -- (10,-0.5) -- cycle;
        \node[] at (9.4,-1.25) {$\set{F}_1$};
        
        \fill[fill = black!30!white] (13.5,-2) -- (15,-0.5) -- (15,-2) -- cycle;
        \node[] at (12.7,-1.25) {$\set{F}_2$};
        
        \draw (10,-4)  circle (4pt) -- (10,-2.5);
        \node[] at (9.4,-3.25) {$\set{F}_3$};
                
        \draw (13.5,-4)  circle (4pt) -- (15,-4);
        \node[] at (12.7,-3.25) {$\set{F}_4$};
    
        \draw (10,-6)  circle (4pt) -- (11.5,-4.5);
        \node[] at (9.4,-5.25) {$\set{F}_5$};
        
        \filldraw (13.5,-6)  circle (4pt);
        \node[] at (12.7,-5.25) {$\set{F}_6$};
        
    \end{tikzpicture}
    
    \caption{Visualisation of nonnegative total variation regularization. The oblique proximal maps the sets $\set{G}_i$ to the corresponding sets $\set{F}_i$.}
    \label{fig:NNTV-penalization}
\end{figure}

Figure \ref{fig:NNTV-penalization} illustrates how the density $\pi_{\set{F}_j}$ is obtained for nonnegative anisotropic total variation in two dimensions, i.e., $f(\vec{x}) = \gamma|x_1 - x_2| + \chi_{\reals^2_+}$. All the mass contained in the $k$-dimensional set $\set{G}_i := \left[\prox_f^{\Sigma^{-1}}\right]^{-1}(\set{F}_i) = (\text{Id} + \Sigma \partial f)(\set{F}_i)$ gets mapped to the $(n-k)$-dimensional set $\set{F}_i$. For the one-dimensional sets $\set{F}_3$, $\set{F}_4$ and $\set{F}_5$, the sets $\Sigma \partial f$ are illustrated by the parallel dashed lines in the corresponding sets $\set{G}_3$, $\set{G}_4$ and $\set{G}_5$. For the zero-dimensional origin $\set{F}_6$, the subdifferential corresponds to the cross-hatched area $\set{G}_6$.

\section{Application to Bayesian linear inverse problems}\label{sec:inverse_problems}
In this section, we will discuss how to apply the concept of a regularized Gaussian and the theory discussed in Section \ref{sec:regularized_distributions} to linear inverse problems. At the beginning of Subsection \ref{subsec:general_theory}, we transformed the randomized regularized linear least squares problem
\begin{equation*}
    \argmin_{\vec{x} \in \reals^n} \left\{\frac{\lambda}{2}\|A\vec{x} - \hat{\vec{b}}\|_2^2 + f(\vec{x})\right\},
\end{equation*}
with $A \in \reals^{m\times n}$ and $\hat{\vec{b}} \sim \mathcal{N}(\vec{b}, \lambda^{-1}I)$ into the proximal framework of Definition \ref{def:regularized_Gaussian}. This transformation required the assumption that $A$ has rank $n$. However, inverse problems often have less data than variables, in which case this transformation cannot be applied. This problem can be addressed by introducing a second randomized linear least squares term, resulting in a problem of the form
\begin{equation}\label{eq:perturbed_regularized_linear_least_squares_2} \argmin_{\vec{x}\in \reals^n}\left\{\frac{\lambda}{2} \|A\vec{x} - \hat{\vec{b}}\|^2_2 + \frac{\delta}{2} \|L\vec{x} - \hat{\vec{c}}\|^2_2 + f(\vec{x})\right\},
\end{equation}
with $\hat{\vec{b}} \sim \mathcal{N}(\vec{b}, \lambda^{-1}I), \hat{\vec{c}} \sim \mathcal{N}(\vec{c}, \delta^{-1}I)$ and assuming $\nullspace(A) \cap \nullspace(L) = \{\vec{0}\}$.
Because \eqref{eq:perturbed_regularized_linear_least_squares_2} is equivalent to
\begin{equation}\label{eq:big_proximal}
    \prox_f^{\Sigma^{-1}}(\vec{x}^\star), \quad \text{with } \Sigma^{-1} = \delta L^TL + \lambda A^TA \text{ and } \vec{x}^\star \sim \mathcal{N}(\Sigma(\lambda A^T\vec{b} + \delta L^T \vec{c}), \Sigma),
\end{equation}
this problem fits into the framework and theory discussed in Section \ref{sec:regularized_distributions}.

In Subsection \ref{subsec:overdetermined_random}, we discuss this setting where the framework of Section \ref{sec:regularized_distributions} applies from a Bayesian viewpoint.
In subsection \ref{subsec:underdetermined_random}, we discuss the setting where the randomized linear least squares term is underdetermined.

\subsection{Overdetermined randomization}\label{subsec:overdetermined_random}

Let us now return to the reconstruction of a signal $\vec{x} \in \reals^n$ from noisy measurements $\vec{b} = A\vec{x} + \vec{e}$, with linear forward operator $A: \reals^n \rightarrow \reals^m$ and independent Gaussian noise $\vec{e} \sim \mathcal{N}(0, \lambda^{-1}I)$. If we assume a priori that
\begin{equation}
    \pi(\vec{x}) \,\propto\, \exp\left(-\frac{\delta}{2}\|L\vec{x} - \vec{c}\|^2_2\right),
\end{equation}
then we can compute the posterior up to scaling, i.e.,
\begin{equation*}
    \pi(\vec{x}\,|\, \vec{b}) \,\propto\, \pi(\vec{b}\,|\, \vec{x})\pi(\vec{x}) \,\propto\, \exp\left(-\frac{\lambda}{2}\|A\vec{x} - \vec{b}\|^2_2-\frac{\delta}{2}\|L\vec{x} - \vec{c}\|^2_2\right).
\end{equation*}

We can now define a regularized posterior $\vec{x}^{\star}$ using the regularized and randomized linear least squares problem \begin{equation*}
    \vec{x}^{\star} = \argmin_{\vec{x}\in \reals^n}\left\{\frac{\lambda}{2} \|A\vec{x} - \hat{\vec{b}}\|^2_2 + \frac{\delta}{2} \|L\vec{x} - \hat{\vec{c}}\|^2_2 + f(\vec{x})\right\}.
\end{equation*}

Under Assumption \ref{as:polyhedral_epigraph}, we can conclude from Theorem \ref{thm:gauss_decomposition}, that the distribution of this regularized posterior conditioned on a set $\set{F}_i$ in the polyhedral partition from Definition \ref{def:polyhedral_partition} satisfies
\begin{equation}
    \pi(\vec{x}\,|\, \set{F}_i, \vec{b}) \,\propto\, \exp\left(-\frac{\lambda}{2}\|A\vec{x} - \vec{b}\|^2_2-\frac{\delta}{2}\|L\vec{x} - \vec{c}\|^2_2 - f(\vec{x})\right) \,\propto\,\pi(\vec{b}\,|\, \vec{x})\pi(\vec{x})\exp(-f(\vec{x})).
\end{equation}

Note that $\pi(\vec{b}\,|\,\vec{x}) = \pi(\vec{b}\,|\,\vec{x}, \set{F}_i)$ for $\vec{x} \in \set{F}_i$, then we can compute the conditional prior corresponding to the regularized posterior, precisely,
\begin{equation}\label{eq:conditional_prior}
    \pi(\vec{x}\,|\,\set{F}_i) \,\propto\, \frac{\pi(\vec{x}\,|\,\set{F}_i, \vec{b})}{\pi(\vec{b}\,|\,\vec{x}, \set{F}_i)} 
    \,\propto\, \pi(\vec{x})\exp(-f(\vec{x})) \,\propto\,
    \exp\left(-\frac{\delta}{2}\|L\vec{x} - \vec{c}\|^2_2 - f(\vec{x})\right).
\end{equation}
Thus, conditioned on the set $\set{F}_i$, the prior is the same as if we would have chosen the prior $\exp\left(-\frac{\delta}{2}\|L\vec{x} - \vec{c}\|^2_2 - f(\vec{x})\right)$ on the whole of $\reals^n$. However, the regularized posterior assigns positive probability to sparse sets. Furthermore, this shows that the regularized posterior corresponds to imposing a variable dimension model as prior on $\vec{x}$, where the models are associated with the sets $\set{F}_i$. Our method therefore implicitly defines the model prior $\pi(\set{F}_i)$, which generally depends on the forward operator and data.

\subsection{Underdetermined randomization}\label{subsec:underdetermined_random}
In this subsection, we consider the setting where the linear inverse problem is under-determined and we do not add an explicit prior, i.e., we consider optimization problems of the form
\begin{equation}\label{eq:underdetermined}
    \argmin_{\vec{x}\in \reals^n}\left\{\frac{\lambda}{2}\|A\vec{x} - \hat{\vec{b}}\|^2_2 + f(\vec{x})\right\},
\end{equation}
where $A \in \reals^{m\times n}$ has rank less than $n$.

If $g(\vec{x}) = f(\vec{x}) - \frac{\delta}{2}\|L\vec{x} - \vec{c}\|$ is still a proper lower semi-continuous convex function for some $L$ such that $\nullspace(A) \cap \nullspace(L) = \{\vec{0}\}$, then
optimization problem \eqref{eq:underdetermined} is equivalent to
\begin{equation}\label{eq:underdetermined_proximal}
    \vec{z}^\star = \argmin_{\vec{x}\in \reals^n}\left\{\frac{\lambda}{2}\|A\vec{x} - \hat{\vec{b}}\|^2_2 + \frac{\delta}{2}\|L\vec{x} - \vec{c}\|^2_2 + g(\vec{x})\right\},
\end{equation}
i.e., the same optimization problem as when we add an explicit prior, but without the randomization of the prior least squares term. Such a case might arise if $f$ is strongly convex.

As a results, $\vec{z}^\star$ can be written as
\begin{equation*}
    \vec{z}^\star = \prox_g^{\lambda A^TA + \delta L^TL}\left(\vec{x}_{\text{deg}}^\star\right), 
\end{equation*}
where $\vec{x}_{\text{deg}}^\star = (\lambda A^TA + \delta L^TL)^{-1}(A^T\hat{\vec{b}} + L^T\vec{c})$ is a degenerate Gaussian distribution.

Because $\prox_f^{\lambda A^TA + \delta L^TL}$ is still a continuous function, we can conclude from Lemma \ref{lemma:pushforward_support} that $\vec{z}^\star$ is not necessarily supported on the whole of $\dom(f)$ and therefore, the support of $\vec{z}^\star$ might not include the true signal $\vec{x}_0 \in \dom(f)$. The following Lemma characterizes which signals can be reconstructed.
\begin{lemma}\label{lem:source_condition}
Assume that $f$ is such that optimization problem \eqref{eq:underdetermined} is well-posed for any perturbation $\hat{\vec{b}}$, i.e., there exists a unique solution that depends continuously on the perturbation $\hat{\vec{b}}$. Then, for any $\vec{x}_0 \in \mathbb{R}^n$, there exists $\vec{q} \in \mathbb{R}^m$ such that
\begin{equation}\label{eq:exact_perturbation}
    \vec{x}_0 = \argmin_{\vec{x}}\left\{\frac{1}{2}\|A\vec{x} - \vec{q}\|_2^2 + f(\vec{x})\right\},
\end{equation}
if and only if $\partial f(\vec{x}_0) \cap \range(A^T) \neq \emptyset$.

Furthermore, if $\hat{\vec{q}} \in \reals^m$ is a random variable such that $\support(\hat{\vec{q}}) = \reals^m$,
then the random variable
\begin{equation}\label{eq:exact_perturbation_random}
    \vec{x}^{\star} = \argmin_{\vec{x}}\left\{\frac{1}{2}\|A\vec{x} - \hat{\vec{q}}\|_2^2 + f(\vec{x})\right\},
\end{equation}
satisfies
\begin{equation*}\label{eq:exact_perturbation_random_support}
    \support(\vec{x}^{\star}) = \closure\left(\left\{\vec{x} \in \dom(f) \,|\, \partial f(\vec{x}) \cap \range(A^T) \neq \emptyset \right\}\right).
\end{equation*}

\end{lemma}
\begin{proof}
The first part follows from the observation that the optimality condition of \eqref{eq:exact_perturbation} can be written as
\begin{equation*}
    A^T(\vec{q}-A\vec{x}_0) \in \partial f(\vec{x}_0).
\end{equation*}

For the second part, by assumption the solution map $\vec{q} \mapsto \argmin_{\vec{x}}\left\{\frac{1}{2}\|A\vec{x} - \vec{q}\|_2^2 + f(\vec{x})\right\}$ is continuous, hence the conclusion follows from the first part combined with Lemma \ref{lemma:pushforward_support}.
\end{proof}

Whilst for general regularization functions $f$, the set $\support(\vec{x}^{\star})$ can be very complicated, the following Corollary describes a few cases where $\support(\vec{x}^{\star})$ is well-behaved.

\begin{corollary}\label{cor:support}
Under Assumption \ref{as:polyhedral_epigraph} consider the polyhedral partition $\{\set{F}_i\}_{i\in \set{J}}$ of Definition \ref{def:polyhedral_partition}.
Assume \eqref{eq:exact_perturbation} is well-posed, then if $\partial f(\vec{x}) \cap \range(A^T) \neq \emptyset$ holds for any $\vec{x} \in \set{F}_i$, then it holds for all $\vec{x} \in \set{F}_i$. Thus, there exists $\tilde{\set{J}} \subset \set{J}$ such that
\begin{equation*}\label{eq:exact_perturbation_random_support_polyhedral}
    \support(\vec{x}^{\star}) = \bigcup_{j \in \tilde{\set{J}}}\closure\left(\set{F}_j\right).
\end{equation*}

Furthermore, $f$ is symmetric and positive homogeneous, then $\support(\vec{x}^{\star})$ is a linear subspace and if $f = \chi_\set{C} + g$, where $\set{C}$ is a closed convex set and $g$ is symmetric and positive homogeneous, then $\support(\vec{x}^{\star})$ is a convex set.
\end{corollary}

As an example, by the second part of Corollary \ref{cor:support}, the support $\support(\vec{x}^{\star})$ is a linear subspace. The convexity of $\support(\vec{x}^{\star})$  guarantees that certain point estimates like the empirical mean stays within the support. However, note that the support does not depend on the data used for reconstruction, because the perturbed data has full support.

Whilst these methods with low-dimensional randomization can still give good results, the low-dimensional support of the distributions cannot guarantee positive mass around the truth, which can be guaranteed by adding the randomized linear least squares regularization as described in the beginning of Subsection \ref{subsec:overdetermined_random}.

\section{Bayesian hierarchical model}\label{sec:bayes_hierarch}

Consider the setting of Subsection \ref{subsec:overdetermined_random} with $\vec{c} = \vec{0}$. To make the model more robust, we will add prior distributions on the hyperparameters $\lambda$ and $\delta$. More precisely, let the conditional distribution $\pi(\vec{x}\,|\,\lambda,\delta,\vec{b})$ be defined by
\begin{equation}\label{eq:regularized_linear_inverse_problem_scaled}
    \argmin_{\vec{x}\in \reals^n}\left\{\frac{\lambda}{2}\|A\vec{x} - \hat{\vec{b}}\|^2_2 + \frac{\delta}{2}\|L\vec{x} - \hat{\vec{c}}\|^2_2 + \sqrt{\delta}f(\vec{x})\right\},
\end{equation}
with $\hat{\vec{b}} \sim \mathcal{N}(\vec{b}, \lambda^{-1} I)$ and $\hat{\vec{c}} \sim \mathcal{N}(\vec{0}, \delta^{-1} I)$.
Furthermore, let $\alpha_\lambda, \alpha_\delta, \beta_\lambda, \beta_\delta > 0$ and define the hyperpriors
\begin{equation*}
    \lambda \sim \Gamma(\alpha_\lambda, \beta_\lambda)\quad \text{and}\quad \delta \sim \Gamma(\alpha_\delta, \beta_\delta), \nonumber\\
\end{equation*}
or equivalently
\begin{equation}\label{eq:hyperpriors}
    \pi(\lambda) \,\propto\, \lambda^{\alpha_\lambda - 1}\exp(-\beta_\lambda \lambda) \quad \text{and}\quad \pi(\delta) \,\propto\, \delta^{\alpha_\delta - 1}\exp(-\beta_\delta \delta).
\end{equation}

To be able to compute the conditional distribution of $\delta$ given $\vec{x}$, we need to know how the conditional prior \eqref{eq:conditional_prior} depends on $\delta$, which is hidden in the proportionality. This proportionality can be computed under some restrictions on the regularization function $f$, as stated in the following lemma.
\begin{lemma}\label{lemma:conditional_prior_scaled}
Assume that the epigraph of $f$ is a polyhedral cone up to an additive constant. Equivalently, $f$ satisfies Assumption \ref{as:polyhedral_epigraph} and is positive homogeneous, i.e., $f(\gamma \vec{x}) = \gamma f(\vec{x})$ for $\gamma \geq 0$. Then the conditional prior $\pi(\vec{x}\,|\, \delta, \set{F}_j)$ satisfies
\begin{equation}\label{eq:conditional_prior_scaled}
    \pi(\vec{x}\,|\, \delta, \set{F}_j)\, \propto\, \delta^{-\dim(\set{F}_j)/2}\exp\left(-\frac{\delta}{2}\|L\vec{x}\|_2^2 - \sqrt{\delta}f(\vec{x})\right).
\end{equation}

\end{lemma}

\begin{proof}
We can obtain the dependence on $\delta$ by computing the normalization constant $K$ using that any set $\set{F}_j$ is invariant under multiplication by a positive constant,
\begin{align*}
    1 = \int_{\set{F}_j}\pi(\vec{x}\,|\, \delta, \set{F}_i) \text{d}\vec{x}
    &= K\int_{\set{F}_j}\exp\left(-\frac{\delta}{2}\|L\vec{x}\|_2^2 - \sqrt{\delta}f(\vec{x})\right) \text{d}\vec{x}\\
    &= K\int_{\set{F}_j}\exp\left(-\frac{1}{2}\|L\sqrt{\delta}\vec{x}\|_2^2 - f(\sqrt{\delta}\vec{x})\right) \text{d}\vec{x}\\
    &= K\delta^{-\text{dim}(\set{F}_j)/2}\int_{\set{F}_j}\exp\left(-\frac{1}{2}\|L\vec{x}\|_2^2 - f(\vec{x})\right) \text{d}\vec{x}.
\end{align*}
Hence we obtain that
\begin{equation*}
    \pi(\vec{x}\,|\, \delta, \set{F}_j)\, \propto\, \delta^{\dim(\set{F}_j)/2}\exp\left(-\frac{\delta}{2}\|L\vec{x}\|_2^2 - \sqrt{\delta}f(\vec{x})\right).
\end{equation*}
\end{proof}

Note that in both optimization problem \eqref{eq:regularized_linear_inverse_problem_scaled} and the conditional prior \eqref{eq:conditional_prior_scaled}, the regularization function is scaled by the hyperparameter $\delta$. While the additional condition that the epigraph of $f$ is a polyhedral cone can be quite restrictive, many important regularization functions do satisfy this condition, e.g., nonnegativity constraints and $l_1$-norm based regularization, like $\|Dx\|_1$ for any matrix $D$.

Using Lemma \ref{lemma:conditional_prior_scaled}, we can obtain the distribution of the (hyper)parameters given $\set{F}_j$,
\begin{align}
    \pi(\vec{x},\lambda,\delta\,|\, \vec{b}, \set{F}_j) \,\propto\, &  \mathbf{1}\left\{\vec{x} \in \set{F}_j\right\}\pi(\vec{b}\,|\, \vec{x},\lambda,\delta) \pi(\vec{x}\,|\,\delta,\set{F}_j) \pi(\lambda) \pi(\delta) \nonumber\\
    \,\propto\, & \mathbf{1}\left\{\vec{x} \in \set{F}_j\right\}\lambda^{m/2+\alpha_{\lambda} - 1}\delta^{\text{dim}(\set{F}_j)/2+\alpha_{\delta} - 1}\nonumber\\
       &\times\exp\left(-\frac{\lambda}{2}\|A\vec{x} - \vec{b}\|_2^2 - \frac{\delta}{2}\|L\vec{x}\|_2^2 - \sqrt{\delta}f(\vec{x}) - \beta_{\lambda}\lambda - \beta_{\delta}\delta\right), \label{eq:full_hyperparameter_dependency}
\end{align}
where $\mathbf{1}\left\{\vec{x} \in \set{F}_j\right\}$ is $1$ if $\vec{x} \in \set{F}_j$ and $0$ otherwise.

Given $\vec{x}$, denote the set $\set{F}$ which contains $\vec{x}$ by $\set{F}(\vec{x})$, then we obtain the following conditional distributions of the hyperparameters,
\begin{align}
    \pi(\lambda\, |\, \vec{x}, \vec{b}) &\,\propto\, \lambda^{m/2 + \alpha_{\lambda}-1}\exp\left( - \frac{\lambda}{2}\|A\vec{x} - \vec{b}\|_2^2-\beta_{\lambda}\lambda\right), \text{ for } \lambda > 0\quad \text{and} \label{pi_lambda}\\
    \pi(\delta\, |\, \vec{x}, \vec{b}) &\,\propto\, \delta^{\text{dim}(\set{F}(\vec{x}))/2 + \alpha_{\delta}-1}\exp\left(- \frac{\delta}{2}\|L\vec{x}\|_2^2 - \sqrt{\delta}f(\vec{x})-\beta_{\delta}\delta \right), \text{ for } \delta > 0,\label{pi_delta}
\end{align}
of which the first conditional distribution satisfies $\lambda\, |\, \vec{x}, \vec{b} \sim \Gamma(m/2 + \alpha_{\lambda}, \frac12\|A\vec{x} - \vec{b}\|_2^2 + \beta_{\lambda})$.

These conditional distribution can be used to derive the Polyhedral Cone Epigraph Hierarchical Gibbs Sampler described in Algorithm \ref{sampler:PCEHGS}.

\begin{algo}[H]
\centering
\begin{minipage}{.6\textwidth}
\begin{algorithmic}[1]
\STATE{ \textbf{Input:} $\vec{x}^0,\alpha_{\lambda}, \beta_{\lambda}, \alpha_{\delta}, \beta_{\delta}, k_{\max}$}

\FOR{$k = 1$ \TO $k_{\max}$ }

\STATE {Compute $(\lambda^k, \delta^k) \sim \pi(\lambda, \delta\,|\,\vec{x},\vec{b})$ as follows:} 
\STATE {\quad$\lambda^k \sim \Gamma\left(m/2+\alpha_{\lambda}, \frac12\|A\vec{x}^{k-1}-\vec{b}\|_2^2 + \beta_{\lambda}\right),$} 
\STATE {\quad%\begin{varwidth}[t]{\linewidth}
    $\delta^k \sim \pi(\delta\,|\,\vec{x}^{k-1},\vec{b})$ defined in \eqref{pi_delta}.} %\,\propto\, \delta^{\text{dim}(\set{F}(\vec{x}^{k-1}))/2+\alpha_{\delta}-1}$}
    %\hspace*{2em}$\times\exp\left(-\frac{\delta}{2} \|L\vec{x}^{k-1}\|_2^2 - \sqrt{\delta}f(\vec{x}^{k-1}) - \beta_{\delta}\delta\right).$
 %\end{varwidth}} 
\STATE {Compute $\vec{x}^k \sim \pi_{\vec{x}|\vec{b},\lambda^k,\delta^k}$ using \eqref{eq:regularized_linear_inverse_problem_scaled}}.

\ENDFOR

\RETURN{$\{(\vec{x}^k,\lambda^k,\delta^k)\}_{k=1,\dots, k_{\max}}$}

\end{algorithmic}
\end{minipage}
\caption{Polyhedral Cone Epigraph Hierarchical Gibbs Sampler for $(\vec{x}, \lambda, \delta)$.}
\label{sampler:PCEHGS}
\end{algo}

Sampling from $\lambda\, |\, \vec{x}, \vec{b}$ is relatively easy compared to $\delta\, |\, \vec{x}, \vec{b}$, due to the former distribution being able to exploit the conjugacy between Gamma and Gaussian distributions. There is one case where the conditional distribution $\delta\, |\, \vec{x}, \vec{b}$ can be simplified to a Gamma distribution. If the regularization function is the characteristic function of a polyhedral cone, then
\begin{equation*}
    \delta\, |\, \vec{x}, \vec{b} \sim \Gamma(\text{dim}(\set{F}(\vec{x}))/2 + \alpha_{\delta}, \frac12\|L\vec{x}\|_2^2 + \beta_{\delta}).
\end{equation*}
For this constrained setting, Algorithm \ref{sampler:PCEHGS} simplifies to the Polyhedral Cone Hierarchical Gibbs Sampler from \cite[Sampler 3.1]{everink2022bayesian}.

An alternative model with simpler conditional distribution is as follows. If $\nullspace(A) = \{\vec{0}\}$, then there is no need for the randomized least squares term $\frac{\delta}{2}\|L\vec{x} - \hat{\vec{c}}\|^2_2$ in optimization problem \eqref{eq:regularized_linear_inverse_problem_scaled} to make use of the proximal framework. Therefore, consider the alternative optimization problem
\begin{equation}\label{eq:regularized_linear_inverse_problem_scaled_alt}
    \argmin_{\vec{x}\in \reals^n}\left\{\frac{\lambda}{2}\|A\vec{x} - \hat{\vec{b}}\|^2_2 + \gamma f(\vec{x})\right\},
\end{equation}
with $\hat{\vec{b}} \sim \mathcal{N}(\vec{b}, \lambda^{-1} I)$. Then we can put a hyperprior on both the noise level $\lambda$ and the regularization strength $\gamma$. More precisely, let
\begin{equation}\label{eq:hyperpriors_alt}
    \lambda \sim \Gamma(\alpha_\lambda, \beta_\lambda)\quad \text{and}\quad \gamma \sim \Gamma(\alpha_\gamma, \beta_\gamma),\\
\end{equation}
with $\alpha_\lambda, \alpha\gamma, \beta_\lambda, \beta_\gamma > 0$.

Using a similar argument as in Lemma \ref{lemma:conditional_prior_scaled}, we obtain the conditional distributions
\begin{align}
    \lambda\, |\, \vec{x}, \vec{b} \sim \Gamma(m/2 + \alpha_{\lambda}, \frac12\|A\vec{x} - \vec{b}\|_2^2 + \beta_{\lambda} ) \label{eq:conditional_lambda}\\
    \gamma\, |\, \vec{x}, \vec{b} \sim \Gamma(\text{dim}(\set{F}(\vec{x})) + \alpha_{\gamma}, f(\vec{x}) + \beta_{\gamma}). \label{eq:conditional_gamma}
\end{align}

Note that these distributions cannot be directly derived from \eqref{eq:full_hyperparameter_dependency} in the case $L = 0$. A Gibbs sampler similar to Algorithm \ref{sampler:PCEHGS} is obtained by replacing the conditional distribution for $\delta$ in line 5 of Algorithm \ref{sampler:PCEHGS} with the conditional distribution \eqref{eq:conditional_gamma}.

Algorithm \ref{sampler:PCEHGS} requires computing $\text{dim}(\set{F}(\vec{x}))$ for any $\vec{x} \in \dom(f)$. This quantity can be computed efficiently for many regularization functions.
Each face $\set{F}$ is contained in the solution space of a linear system $S\vec{x} = \vec{s}$ that describes the sparse subset, furthermore, $\dim(\set{F}) = \dim(\nullspace(S)) = n - \Rank(S)$. For $f(\vec{x}) = \chi_{\reals^n_+}(\vec{x})$ or $f(\vec{x}) = \gamma \|\vec{x}\|_1$, and $\vec{x} \in \reals^n_+$, the matrix $S$ corresponding to $\set{F}(\vec{x})$ consists of the rows of the identity matrix that correspond to the zeroes in $\vec{x}$. Similarly, for $f(\vec{x}) = \gamma \|D\vec{x}\|_1$, the matrix $S$ that corresponds to $\set{F}(\vec{x})$ consists of the rows $\vec{d}$ of $D$ for which $\vec{d}^T\vec{x} = 0$. Although this approach works for any sparsity pattern, the following proposition gives some simple expressions for $\dim(\set{F}(\vec{x}))$ for some of the previous examples.

\begin{proposition}\label{prop:face_dimension}
Let $\gamma > 0$, then the follow expressions for the dimensions of the faces hold,
\begin{equation*}
    \dim(\set{F}(\vec{x})) = 
    \begin{cases}
        \|\vec{x}\|_0,\  & \text{if } f(\vec{x}) = \chi_{\reals_+^n}(\vec{x}) \text{ or } f(\vec{x}) = \gamma \|\vec{x}\|_1,\\
        \|D\vec{x}\|_0+n-\Rank(D),\  & \text{if } f(\vec{x}) = \gamma \|D \vec{x}\|_1,\\
        N_{\text{flat}}(\vec{x}),\  & \text{if } f(\vec{x}) = \gamma \|L\vec{x}\|_1,\\
        N_{\text{nz-flat}}(\vec{x}),\  & \text{if } f(\vec{x}) = \gamma \|L\vec{x}\|_1 + \chi_{\reals_+^n}(\vec{x}),\\
    \end{cases}
    \end{equation*}
    where $D \in \reals^{k\times n}$ with $k \leq n$ and full rank, $L$ is the finite difference matrix for anisotropic total variation in any finite dimension, $N_{\text{flat}}(\vec{x})$ is the number of connected flat regions in $\vec{x}$ and $N_{\text{nz-flat}}(\vec{x})$ is the number of non-zero connected flat regions in $\vec{x}$. 
\end{proposition}

\begin{proof}
For $f(\vec{x}) = \chi_{\reals_+^n}(x)$ or $f(\vec{x}) = \gamma \|x\|_0$, the matrix $S$ corresponding to $\set{F}(\vec{x})$ has rows $\vec{e}_i$ for $i$ such that $\vec{x}_i = 0$. Hence, the rank of $S$ is the number of zero elements in $\vec{x}$, i.e., $\Rank(S) = n-\|\vec{x}\|_0$, therefore, $\dim(\set{F}(\vec{x})) = \|\vec{x}\|_0$.

Similarly for $f(\vec{x}) = \gamma \|D \vec{x}\|_1$, the rows of $S$ corresponding to $\set{F}(\vec{x})$ are the rows $\vec{d}$ of $D$ for which $\vec{d}^T\vec{x} = 0$. Because $D$ is full rank, $S$ is full-rank with $\Rank(S) = \Rank(D)-\|D\vec{x}\|_0$.

For $f(\vec{x}) = \gamma \|L\vec{x}\|_1$, let $\vec{x}\in \reals^n$ and $P$ be a permutation matrix such that $P\vec{x}$ can be split into $N_{\text{flat}}(\vec{x})$ consecutive components, each corresponding to a connected flat area in $\vec{x}$. The matrix $S$ corresponding to $\set{F}(\vec{x})$ has a row for each neighboring values of $\vec{x}$ with the same value, hence $PSP^T$ is a block matrix with $N_{\text{flat}}(\vec{x})$ rectangular blocks on the diagonal and zero otherwise. If $l_j$ is the number of columns in a block $B_j$, i.e., $l_j$ is the number of values in the flat component that $B_j$ represents, then $\Rank(B_j) = l_j-1$. Therefore,
\begin{equation*}
    \Rank(S) = \Rank(PSP^T) = \sum_{j=1}^{N_{\text{flat}}(\vec{x})}\Rank(B_j) = \sum_{j=1}^{N_{\text{flat}}(\vec{x})}(l_j-1) = n - N_{\text{flat}}(\vec{x}),
\end{equation*}
hence we can conclude that $\dim(\set{F}(\vec{x})) = N_{\text{flat}}(\vec{x})$.

The argument for $f(\vec{x}) =  \gamma \|D\vec{x}\|_1 + \chi_{\reals_+^n}(\vec{x})$ is almost identical, except that the rank for any block $B_j$ of zero components satisfies $\Rank(B_j) = l_j$, hence $\Rank(S) = n-N_{\text{nz-flat}}(\vec{x})$.
\end{proof}

Although Proposition \ref{prop:face_dimension} states that $\dim(\set{F}(\vec{x}))$ can be efficiently computed, computing samples from \eqref{eq:regularized_linear_inverse_problem_scaled} given the hyperparameters is a problem considered in Section \ref{sec:numerical_examples}.

\section{Numerical examples}\label{sec:numerical_examples}
We now apply the methods of Sections \ref{sec:inverse_problems} and \ref{sec:bayes_hierarch} to a deblurring and a computed tomography example. Subsection \ref{subsec:numex_reg} focuses on the effects of regularization and Subsection \ref{subsec:numex_gibbs} focuses on the Gibbs sampler from Section \ref{sec:bayes_hierarch}. Finally, in Subsection \ref{subsec:numex_CT}, the ideas of Subsection \ref{subsec:underdetermined_random} are used to show that valid results can be obtained in the underdetermined setting. For all experiments, the Alternating Direction Method of Multipliers (ADMM) is used for approximately solving \eqref{eq:regularized_linear_inverse_problem_scaled} with a fixed number of iterations. A short explanation of the algorithm can be found in Appendix \ref{ap:algorithm}.

\subsection{Effect of regularization}\label{subsec:numex_reg}
In this subsection and the next, we consider a one-dimensional Gaussian deblurring problem defined by
\begin{equation*}
    \vec{b} = A\vec{x} + \vec{e},
\end{equation*}
for a true signal $\vec{x}\in \reals^n$, noise $\vec{e} \sim \mathcal{N}(\vec{0}, \lambda^{-1}I)$ with hyperparameter $\lambda = 1000$ and forward operator $A$ defined by the Toeplitz matrix
\begin{equation*}
    A_{ij} = \frac{h}{\sigma\sqrt{2\pi}}\exp\left(-\frac12 \left(\frac{h(i-j)}{\sigma}\right)^2\right),\quad \text{ for } i,j = 1,\dots, n,
\end{equation*}
where $m=n=128$, $h=1/n$ and $\sigma = 0.02$. The true signal $\vec{x}$and the noisy data $\vec{b}$ are shown in Figure \ref{fig:imp_ex_data}.

\begin{figure}
    \centering
    \includegraphics[width=0.8\textwidth]{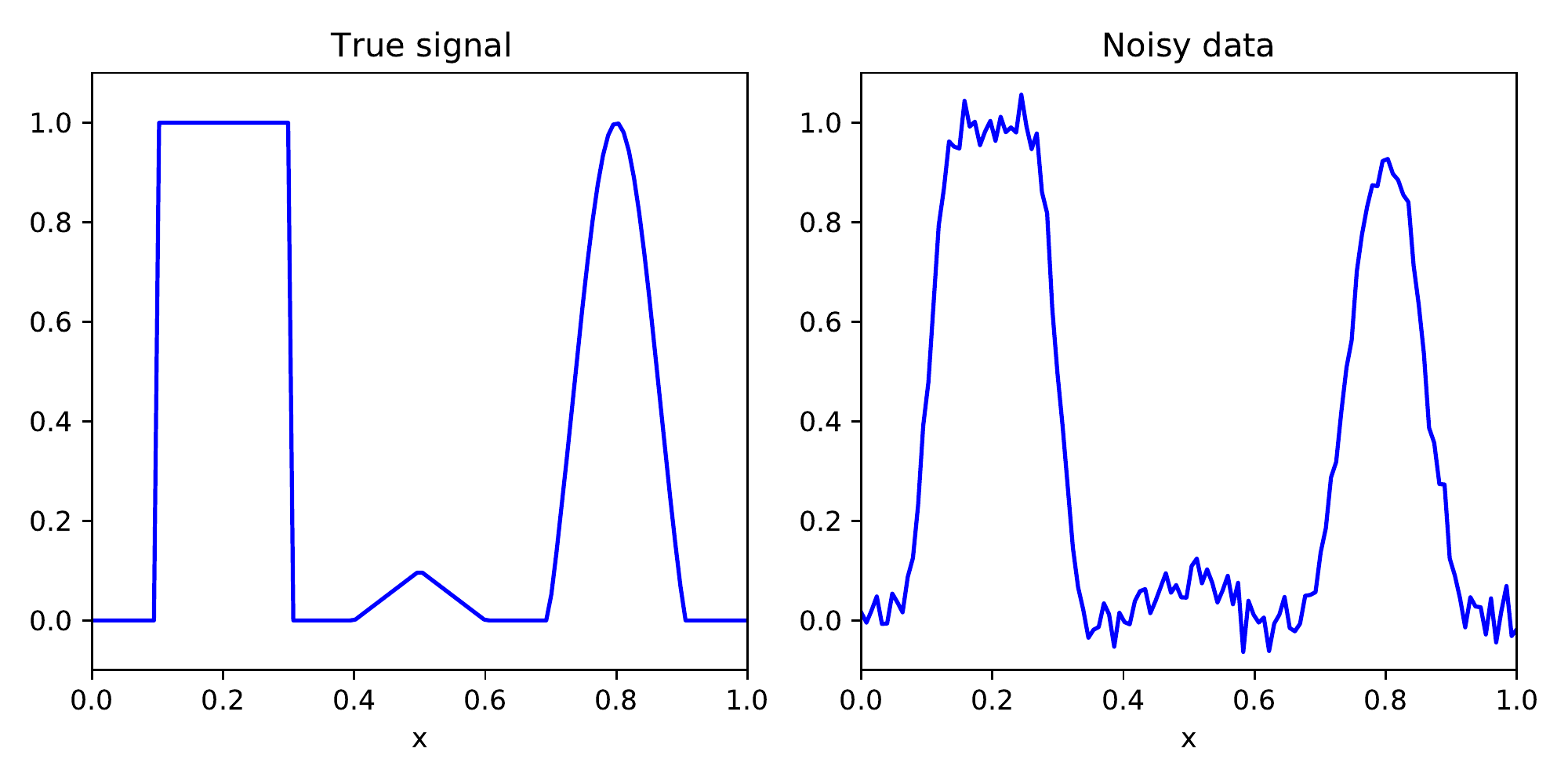}
    \caption{True signal (left) and noisy data with $\lambda = 1000$ (right).}
    \label{fig:imp_ex_data}
\end{figure}

The forward operator $A$ is a Gaussian blur operator with zero boundary condition, and $A$ is full rank. Therefore, we do not require an additional randomized least squares term to apply the theory of Subsection \ref{subsec:overdetermined_random} and Section \ref{sec:bayes_hierarch}, and will consider the randomized optimization problem
\begin{equation}\label{eq:experiment1_optimization}
    \argmin_{\vec{x} \in \reals^{128}} \left\{\frac{\lambda}{2}\|A\vec{x}- \hat{\vec{b}}\|_2^2 + \gamma \|L\vec{x}\|_1\right\}
\end{equation}
with $\hat{\vec{b}} \sim \mathcal{N}(\vec{b}, \lambda^{-1}I)$ and $\gamma > 0$, where $L \in \reals^{(n-1)\times n}$ is a finite difference matrix. For the unperturbed problem, i.e., replacing $\hat{\vec{b}}$ by $\vec{b}$, the deterministic solution corresponds to the MAP estimate of the posterior obtained by using a Laplace difference prior. This posterior is of the form
\begin{equation}\label{eq:laplace_difference}
    \pi(\vec{x}\,|\,\vec{b})\, \propto\, \exp\left(-\frac{\lambda}{2}\|A\vec{x} - \vec{b}\|_2^2 - \gamma \|L\vec{x}\|_1\right).
\end{equation}

From Subsection \ref{subsec:overdetermined_random}, we know that samples from the implicitly defined probability distribution \eqref{eq:experiment1_optimization} have piece-wise constant behavior with positive probability. Unlike samples from the explicitly defined posterior distribution \eqref{eq:laplace_difference}. Figure  \ref{fig:imp_ex_samples} shows few (nearly) independent samples of the implicit distribution \eqref{eq:experiment1_optimization} and explicit distribution \eqref{eq:laplace_difference} for different values of $\gamma$, i.e., the regularization strength and prior strength respectively. The samples for the explicit distribution \eqref{eq:laplace_difference} have been computed using a Random Walk MCMC algorithm, whilst samples from the implicit distribution  \eqref{eq:experiment1_optimization} have been computed using ADMM. Figure \ref{fig:imp_ex_samples} highlights the difference in regularity between the samples from a continuous distribution and those obtained through a varying dimension model.

\begin{figure}
    \centering
    \includegraphics[width=\textwidth]{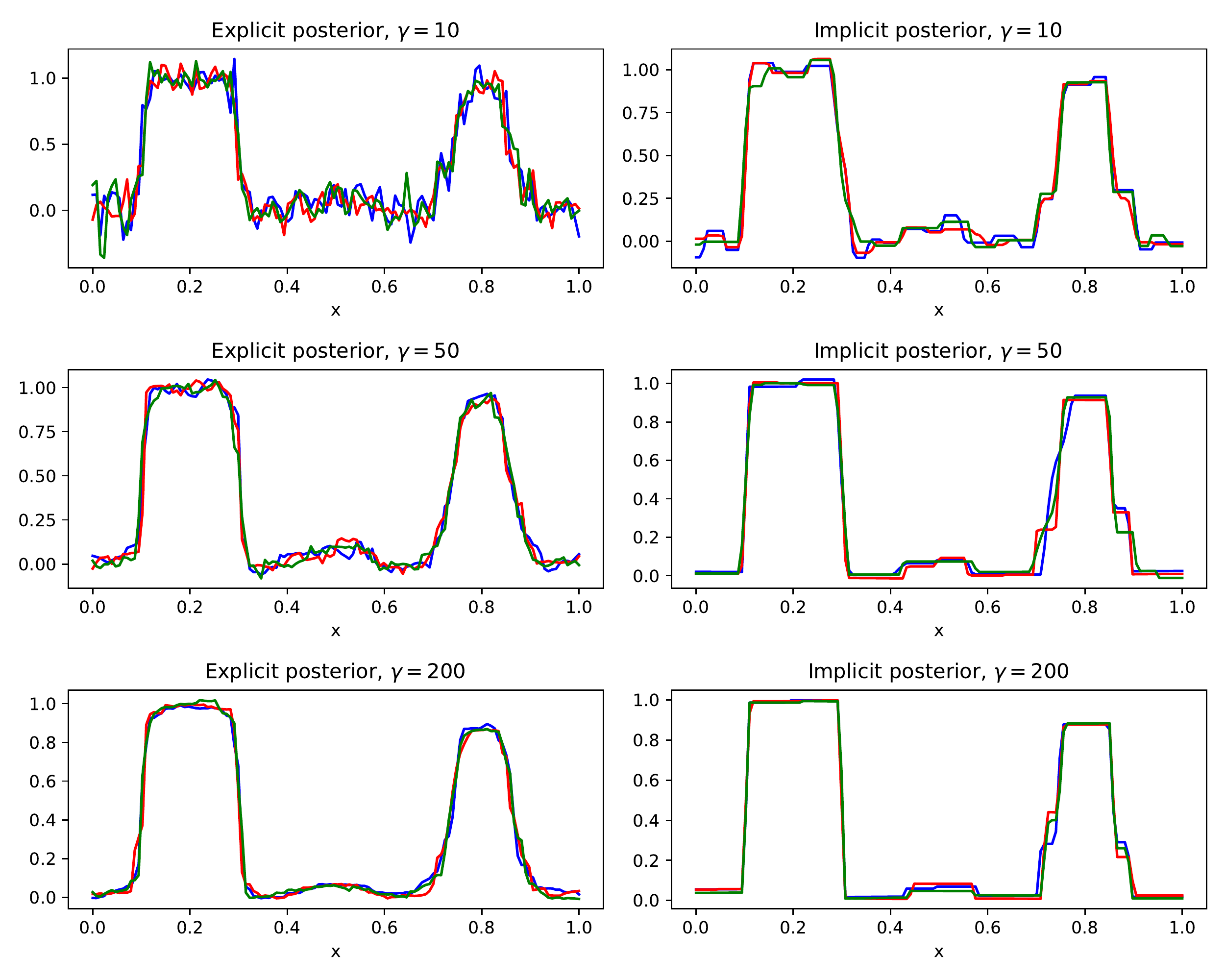}
    \caption{Samples from a posterior with a Laplace difference prior (left) and from a total variation regularized randomized linear least squares problem (right).}
    \label{fig:imp_ex_samples}
\end{figure}

Whilst the MAP estimate of \eqref{eq:laplace_difference} corresponds to the unperturbed solution of \eqref{eq:experiment1_optimization}, other point estimates show different behavior. Figure \ref{fig:imp_ex_median} shows the componentwise median for both models for different values of $\gamma$. The explicit model \eqref{eq:laplace_difference} shows a higher degree of smoothness than the implicit model \eqref{eq:experiment1_optimization}.

\begin{figure}
    \centering
    \includegraphics[width=\textwidth]{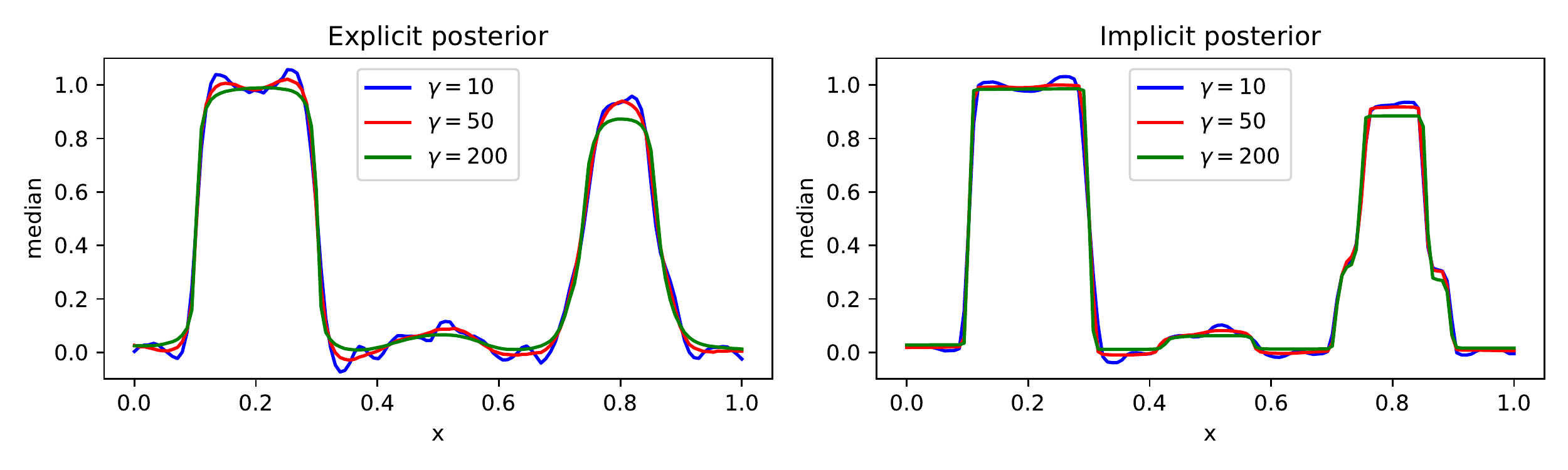}
    \caption{Median of posterior with a Laplace difference prior (left) and of a total variation regularized randomized linear least squares problem (right).}
    \label{fig:imp_ex_median}
\end{figure}

Figure \ref{fig:imp_ex_CI} shows the widths of $95\%$-credible intervals for both models. For the same parameter $\gamma$, the regularized distribution has lower uncertainty, but both models share similar features. Both models give similar results for the discontinuous jump on the left part of the signal, but the implicit model shows more noticeable peaks around the small peak in the center of the signal. Although both models give high uncertainty to the sides of the parabola in the right side of the signal, the uncertainty is more uniform for the implicit model, possibly due to the uncertainty in jumps as shown in Figure \ref{fig:imp_ex_samples}.

\begin{figure}
    \centering
    \includegraphics[width=\textwidth]{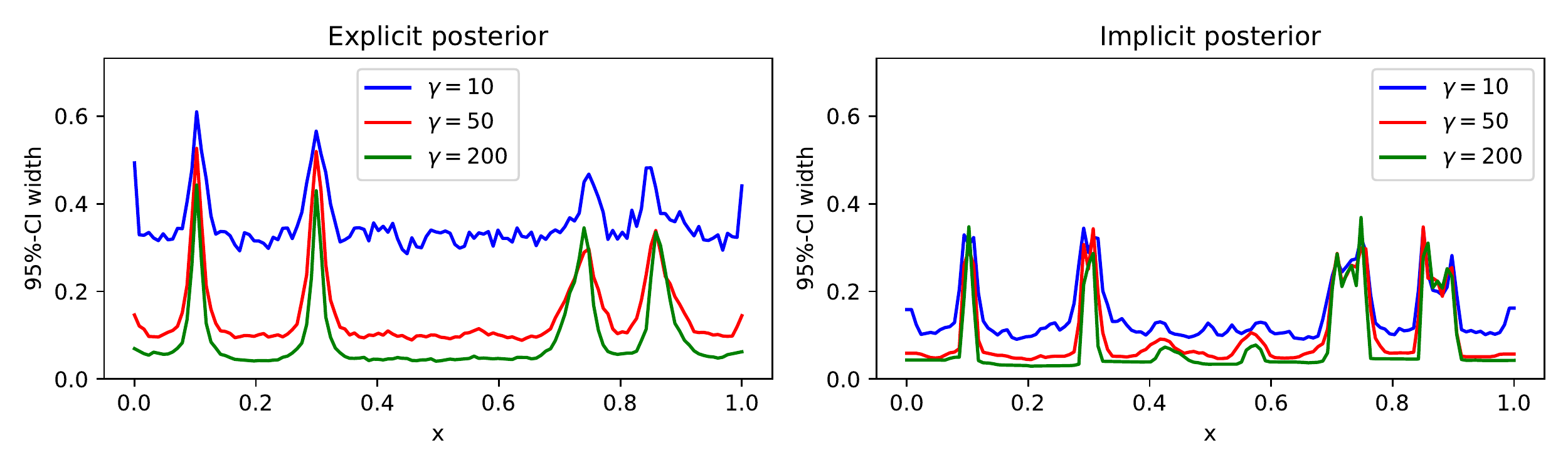}
    \caption{Width of 95\% credible intervals of a posterior with a Laplace difference prior (left) and of a total variation regularized randomized linear least squares problem (right).}
    \label{fig:imp_ex_CI}
\end{figure}

\subsection{Gibbs sampler}\label{subsec:numex_gibbs}

The same deblurring problem can also be solved used the Gibbs sampler developed in Section \ref{sec:bayes_hierarch}. Specifically, we consider the nonnegativity constrained and total variation regularized problem
\begin{equation}\label{eq:experiment2_optimization}
    \argmin_{\vec{x} \in \reals^{128}_+} \left\{\frac{\lambda}{2}\|A\vec{x}- \hat{\vec{b}}\|_2^2 + \gamma \|L\vec{x}\|_1\right\},
\end{equation}
with $\vec{e} \sim \mathcal{N}(0, \lambda^{-1}I)$ and hyperpriors $\lambda \sim \Gamma(1, 10^{-4})$ and $\gamma \sim \Gamma(1, 10^{-4})$. Samples from this hierarchical Bayesian model are obtained using a version of Algorithm \ref{sampler:PCEHGS}. Due to the cost of computing a sample, we investigate solving \eqref{eq:experiment2_optimization} approximately. The disadvantage is that samples describe a slightly different distribution, but also that inaccurate solution could have a big impact on the conditional distributions \eqref{eq:conditional_lambda} and \eqref{eq:conditional_gamma}. Figure \ref{fig:Gibbs_distributions0} shows the distributions of the hyperparameters $\lambda$ and $\gamma$, the distribution of the regularization strength $\gamma/\lambda$, and the distribution of the sparsity of the samples expressed as $\dim(\set{F(\vec{x})})$. Note that the distribution of $\lambda$ is approximately independent of the number of iterations. This can be explained by ADMM being able to minimize the objective function to low accuracy rather quickly, therefore obtaining $\|A\vec{x} - \vec{b}\|_2^2$ up to an order of magnitude. The other chains do decay quite rapidly and already give consistent results after relatively few iterations. The main issue is that ADMM requires more iterations to solve the optimization problem accurately, which is needed for an accurate computation of $\dim(\set{F(\vec{x})})$. However, due to the small size of the problem, it quickly converges to the right order of magnitude, which is enough to get stable regularization effects.

\begin{figure}
    \centering
    \includegraphics[width=\textwidth]{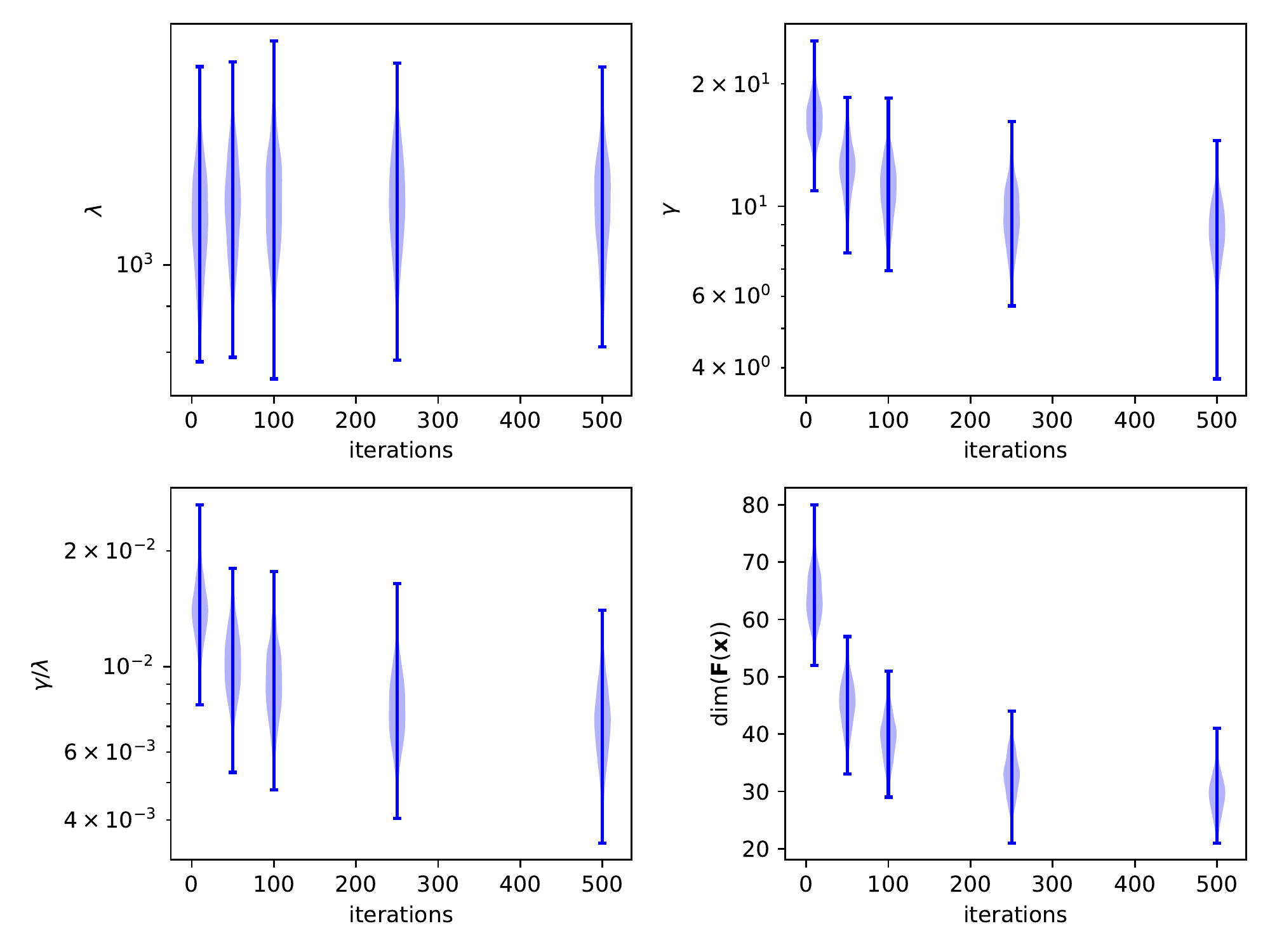}
    \caption{Distribution of the hyperparameters $\lambda$ and $\gamma$, the regularization strength $\gamma/\lambda$ and the sparsity $\dim(\set{F(\vec{x})})$ for increasing accuracy/iterations.}
    \label{fig:Gibbs_distributions0}
\end{figure}

Figure \ref{fig:Gibbs_CI} shows that ADMM quickly achieves low accuracy, but requires a lot more iterations for high accuracy, which shows the width of component-wise $95\%$ credible intervals. Although $50$ iterations seems to be enough to obtain the areas of relatively high uncertainty, peaks corresponding to large jumps, as can be seen in Figure \ref{fig:imp_ex_samples}, keep growing significantly. Such accuracy is, however, not always necessary in applications.

\begin{figure}
    \centering
    \includegraphics[width=0.8\textwidth]{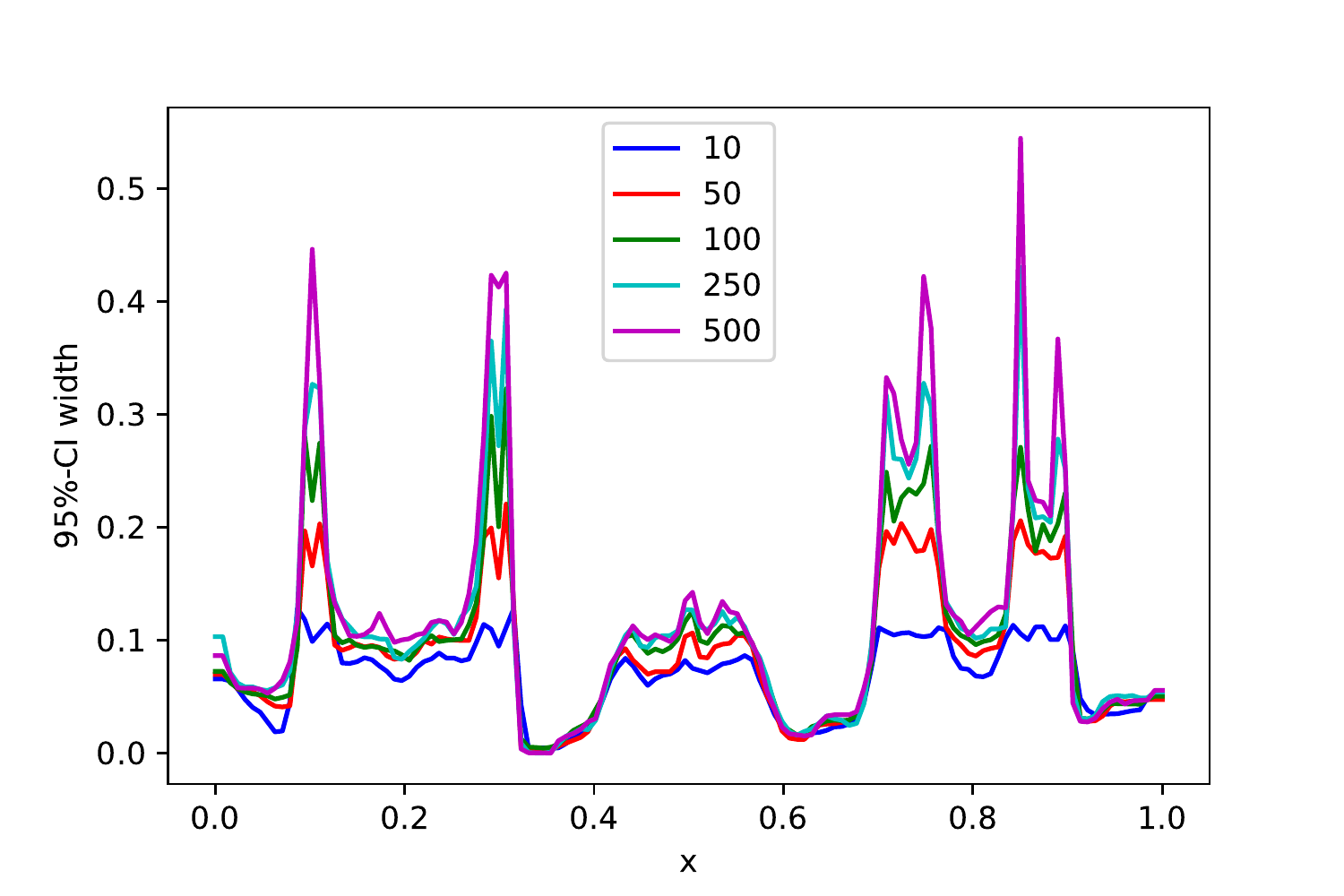}
    \caption{Width of 95\% credible intervals at increasing accuracy/iterations.}
    \label{fig:Gibbs_CI}
\end{figure}

The value of $\gamma/\lambda$ greatly impacts the regularity of the samples and thereby the measured uncertainty. Figure \ref{fig:Gibbs_distributions} shows the mean regularization strength $\gamma/\lambda$ as function of the expectation and variance of the gamma prior put on $\gamma$. The figure shows that, similarly to $\lambda$, as long as the variance is large enough, the $\gamma/\lambda$ chain is stable. For variance too small, the expectation dominates $\gamma$, thereby resulting in an unstable hyperprior.

\begin{figure}
    \centering
    \includegraphics[width=0.5\textwidth]{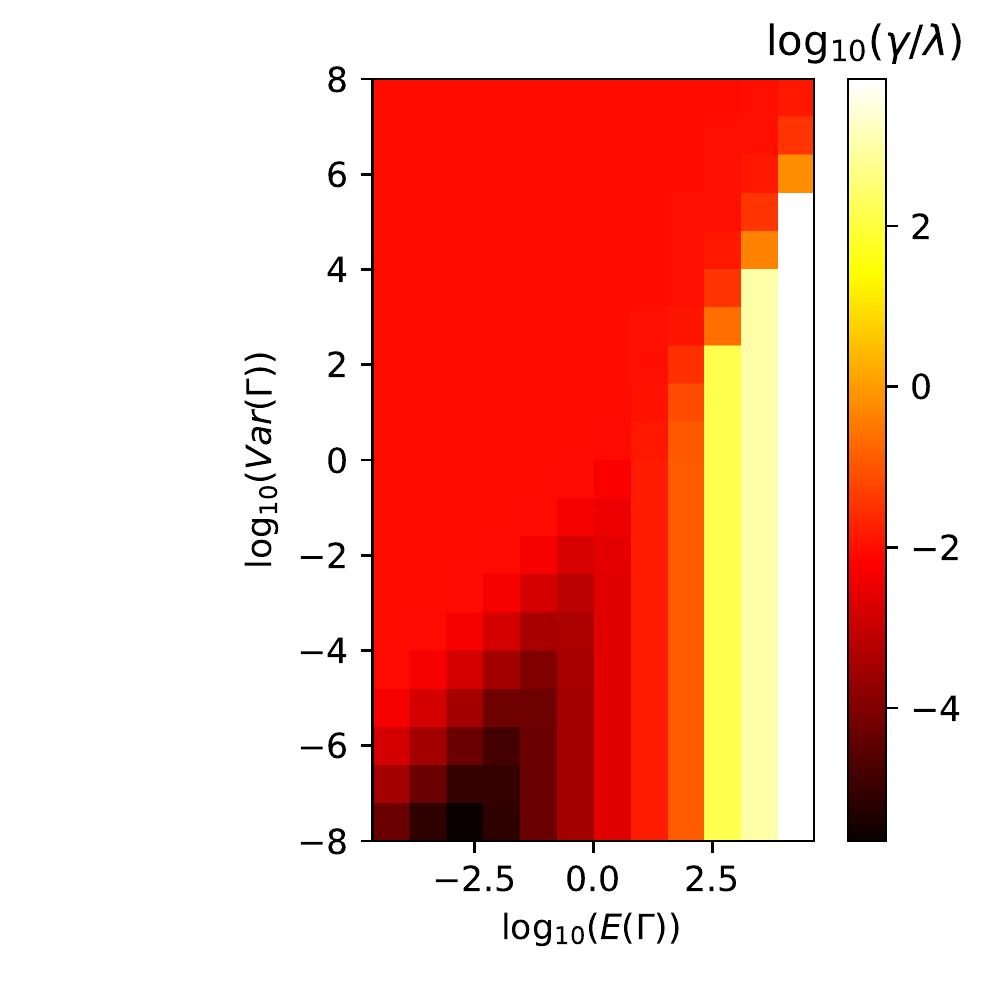}
    \caption{The average regularization strength $\gamma/\lambda$ as functions of the expectation and variance of the hyperprior on $\gamma$.}
    \label{fig:Gibbs_distributions}
\end{figure}

\subsection{Underdetermined linear system}\label{subsec:numex_CT}
Consider a small scale CT problem \cite{hansen2021computed} of the form
\begin{equation*}
    \vec{b} = A\vec{x} + \vec{e},
\end{equation*}
where the true signal $\vec{x} \in \reals^{100\times 100}$ is the Shepp-Logan phantom and the noise satisfies $\vec{e} \sim \mathcal{N}(0, \lambda^{-1}I)$ with $\lambda = 10$. Furthermore, the forward operator $A$ is a discretized Radon transform at 20 equally spaced angles from 0 to 180 degrees with 120 rays per angle and using parallel-beam geometry. The operator was generated using AIR Tools II \cite{hansen2018air}. The true signal $\vec{x}$ and the noisy sinogram $\vec{b} \in \reals^{20 \times 120}$ are shown in Figure \ref{fig:CT_data}.

\begin{figure}
    \centering
    \includegraphics[width=\textwidth]{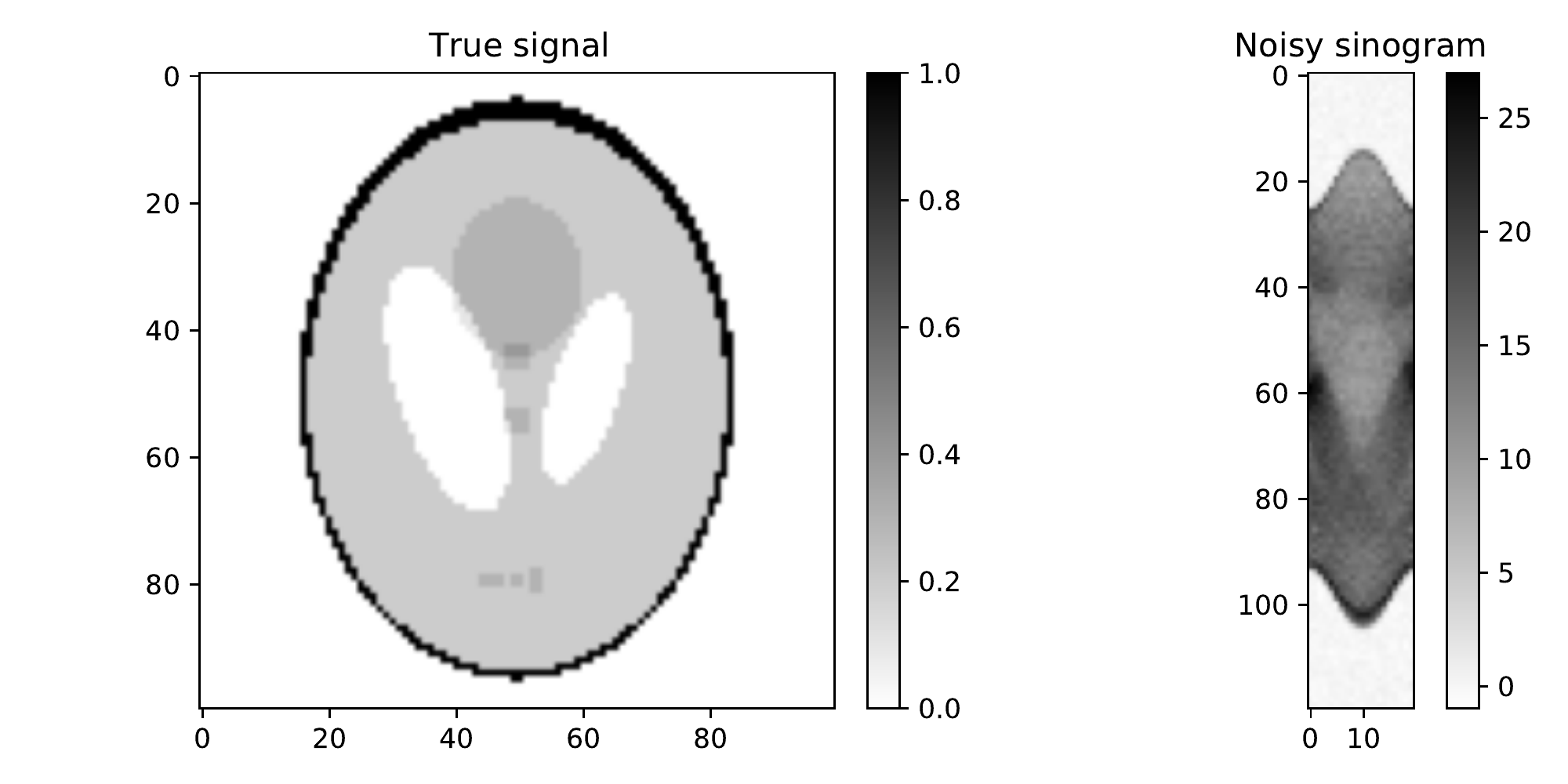}
    \caption{Shepp-Logan phantom (left) and noisy sinogram with $\lambda = 10$ (right).}
    \label{fig:CT_data}
\end{figure}

Note that this is a sparse-angle tomography problem with $10.000$ variables and only $2.400$ measurements. To solve the CT problem and quantify uncertainty, we use the framework with underdetermined randomization as described in Subsection \ref{subsec:underdetermined_random}. More precisely, we consider the probability distribution defined by
\begin{equation*}
    \argmin_{\vec{x} \in \reals^{100\times100}_+} \left\{\frac{\lambda}{2}\|A\vec{x}- \hat{\vec{b}}\|_2^2 + \gamma \|L\vec{x}\|_1\right\},
\end{equation*}
with $\hat{\vec{b}} \sim \mathcal{N}(\vec{b}, \lambda^{-1}I)$, $\gamma > 0$ and $L$ is a two-dimensional finite difference operator. Thus, we solve a randomized non-negativity constrained linear least squares problem with isotropic total variation regularization.

Figure \ref{fig:CT_results} shows the median and 95\%-credible interval widths for $\gamma = 10, 50$ and $ 100$, as computed from $500$ independent samples computed through ADMM for $200$ iterations with $\rho = 200$. Note that for all the choices of $\gamma$, the general shape has been properly reconstructed, but for the larger $\gamma$, some details have been nearly completely smoothed out. This is especially noticeable in the lack of uncertainty in the credible interval widths. By Corollary \ref{cor:support}, the support of the probability distribution is independent of the choice of $\gamma$. However, as expected from total variation regularization, larger $\gamma$ will reduce the probability of small details.

\begin{figure}
    \centering
    \includegraphics[width=\textwidth]{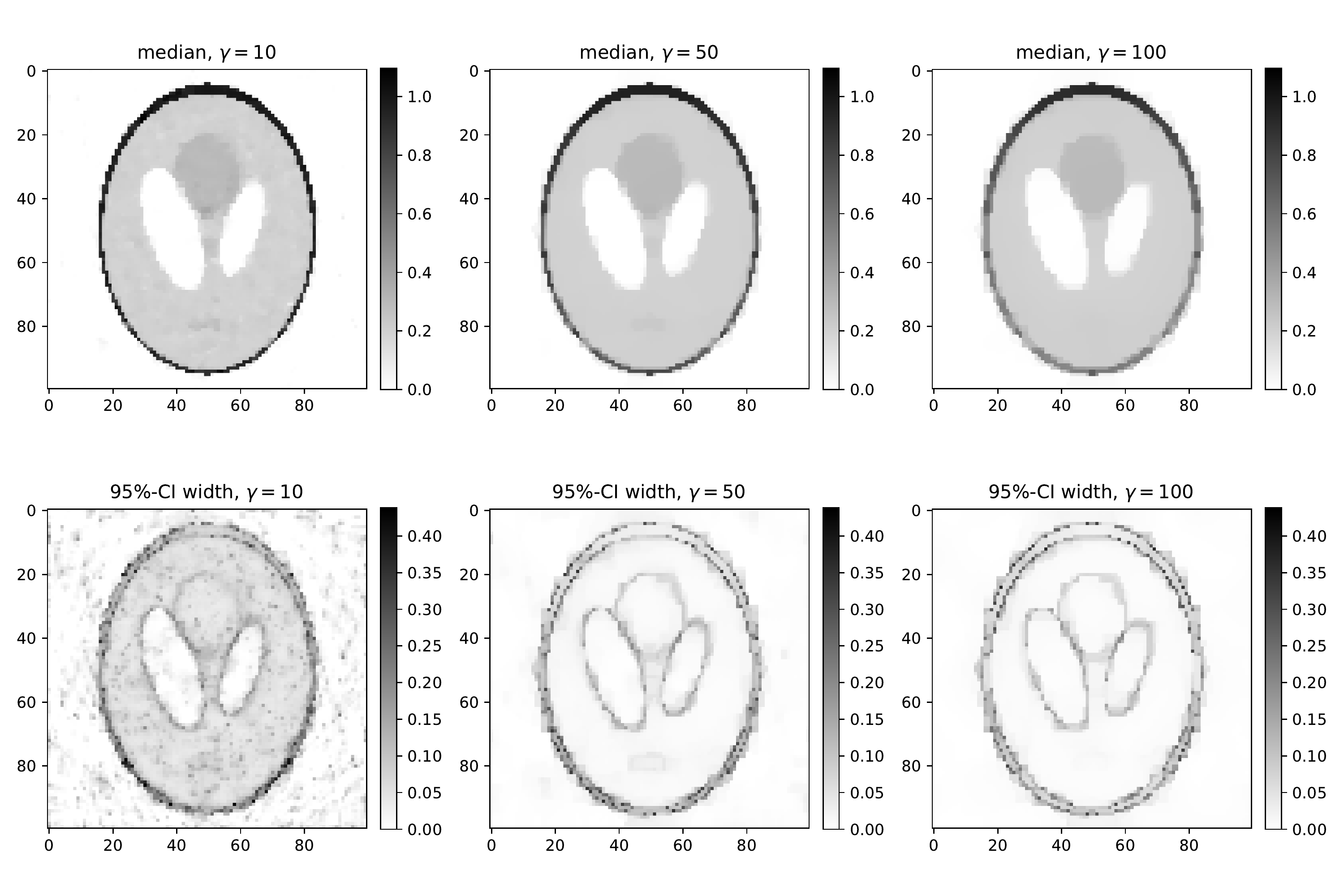}
    \caption{Component-wise median and width of 95\% credible intervals for three different regularization strengths $\gamma$.}
    \label{fig:CT_results}
\end{figure}

\newpage
\section{Conclusions and Further Research}
We have presented a framework for combining regularization with Gaussian distributions. By repeatedly solving a regularized linear least squares for different Gaussian perturbations of the data, we obtain samples from a probability distribution that is equivalent to post-processing a Gaussian posterior with a proximal operator. If the regularization function is sufficiently smooth, then this distribution is a continuous distribution. However, for sparsity inducing regularization functions, e.g., total variation and lasso, the regularized distribution assigns positive probability to low-dimensional subspaces corresponding to the sparsity modeled with the regularization function. Such distributions behave similarly to varying dimension models, but are implicitly defined by the solution of a randomized optimization problem.

We applied the theory to Bayesian linear inverse problems and discussed how the regularized Gaussian can then be interpreted as an implicitly defined posterior distribution by looking into the underlying prior distribution. This implicit prior perspective allowed us to derive Gibbs samplers for some Bayesian hierarchical models that include the regularized Gaussian based on regularization functions whose epigraph are polyhedral cones, e.g. non-negative total variation. We also provided some results on the support of the distribution in the general case where the linear least squares term is underdetermined.

Because sampling from this distribution only requires solving regularized linear least squares problems, we only need tools from optimization theory in order to obtain samples. For some small scale deblurring and computed tomography examples, we showed the difference between our framework and continuous distributions, and highlighted the robustness of the Gibbs sampler for the Bayesian hierarchical model.

Obtaining a single independent sample from the regularized distribution has approximately the same computational cost as solving a similar regularized variational inverse problem. It is therefore computationally expensive to obtain many exact independent samples. However, we have shown in a small numerical  example that inaccurate samples can still give good results at a fraction of the computational cost. To make the approach more practical, further research could include a more in-depth study on the effect that inaccurately solving the optimization problem has on the distribution of the samples.

In this work, we have not considered all possible regularization functions. A big part of the theory focused on low-dimensional subspaces, for which we only provided results for sufficiently smooth functions or non-differentiable functions with sufficiently well-behaved subdifferential. Furthermore, most theory provided relies on the linear least squares term to be strongly convex, providing only some simple properties for the general case. Further investigating these and more general properties of regularized distributions is a topic for further study.

\bibliographystyle{abbrv}
\bibliography{references}

\begin{thebibliography}{10}

\bibitem{bauschke2011convex}
H.~H. Bauschke, P.~L. Combettes, et~al.
\newblock {\em Convex analysis and monotone operator theory in Hilbert spaces},
  volume 408.
\newblock Springer, 2011.

\bibitem{boyd2011distributed}
S.~Boyd, N.~Parikh, E.~Chu, B.~Peleato, J.~Eckstein, et~al.
\newblock Distributed optimization and statistical learning via the alternating
  direction method of multipliers.
\newblock {\em Foundations and Trends{\textregistered} in Machine learning},
  3(1):1--122, 2011.

\bibitem{chen2016direct}
C.~Chen, B.~He, Y.~Ye, and X.~Yuan.
\newblock The direct extension of admm for multi-block convex minimization
  problems is not necessarily convergent.
\newblock {\em Mathematical Programming}, 155(1):57--79, 2016.

\bibitem{condat2013direct}
L.~Condat.
\newblock A direct algorithm for 1-d total variation denoising.
\newblock {\em IEEE Signal Processing Letters}, 20(11):1054--1057, 2013.

\bibitem{everink2022bayesian}
J.~M. Everink, Y.~Dong, and M.~S. Andersen.
\newblock Bayesian inference with projected densities.
\newblock {\em arXiv preprint arXiv:2209.12481}, 2022.

\bibitem{ewald2020distribution}
K.~Ewald and U.~Schneider.
\newblock On the distribution, model selection properties and uniqueness of the
  lasso estimator in low and high dimensions.
\newblock {\em Electronic Journal of Statistics}, 14(1), 2020.

\bibitem{foucart2013compressive}
S.~Foucart and R.~Holder.
\newblock {\em A Mathematical Introduction to Compressive Sensing}.
\newblock Birkhäuser, 2013.

\bibitem{green1995reversible}
P.~J. Green.
\newblock Reversible jump {Markov} chain {Monte Carlo} computation and
  {Bayesian} model determination.
\newblock {\em Biometrika}, 82(4):711--732, 1995.

\bibitem{hansen2010discrete}
P.~C. Hansen.
\newblock {\em Discrete inverse problems: insight and algorithms}.
\newblock SIAM, 2010.

\bibitem{hansen2021computed}
P.~C. Hansen, J.~J{\o}rgensen, and W.~R. Lionheart.
\newblock {\em Computed Tomography: Algorithms, Insight, and Just Enough
  Theory}.
\newblock SIAM, 2021.

\bibitem{hansen2018air}
P.~C. Hansen and J.~S. J{\o}rgensen.
\newblock Air tools ii: algebraic iterative reconstruction methods, improved
  implementation.
\newblock {\em Numerical Algorithms}, 79(1):107--137, 2018.

\bibitem{orieux2012sampling}
F.~Orieux, O.~F{\'e}ron, and J.-F. Giovannelli.
\newblock Sampling high-dimensional gaussian distributions for general linear
  inverse problems.
\newblock {\em IEEE Signal Processing Letters}, 19(5):251--254, 2012.

\bibitem{parikh2014proximal}
N.~Parikh, S.~Boyd, et~al.
\newblock Proximal algorithms.
\newblock {\em Foundations and trends{\textregistered} in Optimization},
  1(3):127--239, 2014.

\bibitem{robert1999monte}
C.~P. Robert and G.~Casella.
\newblock {\em {Monte Carlo} statistical methods}, volume~2.
\newblock Springer, 1999.

\bibitem{rockafellar2009variational}
R.~T. Rockafellar and R.~J.-B. Wets.
\newblock {\em Variational analysis}, volume 317.
\newblock Springer Science \& Business Media, 2009.

\bibitem{rudin1987real}
W.~Rudin.
\newblock {\em Real and Complex Analysis, 3rd Ed.}
\newblock McGraw-Hill, Inc., USA, 1987.

\bibitem{schreck2015shrinkage}
A.~Schreck, G.~Fort, S.~Le~Corff, and E.~Moulines.
\newblock A shrinkage-thresholding {Metropolis} adjusted {Langevin} algorithm
  for {Bayesian} variable selection.
\newblock {\em IEEE Journal of Selected Topics in Signal Processing},
  10(2):366--375, 2015.

\bibitem{waldmann2014topology}
S.~Waldmann.
\newblock {\em Topology: an introduction}.
\newblock Springer, 2014.

\bibitem{wang2017bayesian}
Z.~Wang, J.~M. Bardsley, A.~Solonen, T.~Cui, and Y.~M. Marzouk.
\newblock Bayesian inverse problems with {$l_1$} priors: a
  randomize-then-optimize approach.
\newblock {\em SIAM Journal on Scientific Computing}, 39(5):S140--S166, 2017.

\end{thebibliography}

\appendix
\section{Algorithm} \label{ap:algorithm}

Consider solving the optimization problem
\begin{equation}\label{eq:ap_solve}
    \argmin_{\vec{x} \in \reals^n}\left\{\frac{1}{2}\|A\vec{x} - \vec{b}\|_2^2 + \sum_{i=1}^{k} f_i(L_i \vec{x})\right\},
\end{equation}
where the functions $f_i : \reals^{c_i} \rightarrow \reals$ are proper lower semi-continuous functions for which the proximal operator $\prox_{f_i}^{\sigma}$ is efficiently computable.

Solving \eqref{eq:ap_solve} repeatedly for different functions $f_i$ is the main computational problem for the methods presented in this work. For the numerical experiments in Section \ref{sec:numerical_examples}, we have used a variant of the Alternating Direction Method of Multipliers (ADMM) \cite{boyd2011distributed}. For this, rewrite the unconstrained problem \eqref{eq:ap_solve} to the following constrained problem:
\begin{align}\label{eq:ap_solve_alt}
    \text{Minimize }&\left\{\frac{1}{2}\|A\vec{x} - \vec{b}\|_2^2 + \sum_{i=1}^{k} f_i(\vec{y}_i)\right\},\\
    \text{subject to }& 
    \begin{bmatrix}
    L_1 \\
    \vdots \\
    L_k
    \end{bmatrix} \vec{x} - \begin{bmatrix}
    \vec{y}_1 \\
    \vdots \\
    \vec{y}_k
    \end{bmatrix} = \vec{0}. \nonumber
\end{align}

The ADMM algorithm corresponding to \eqref{eq:ap_solve_alt} requires the evaluation of the proximal operator for $\sum_{i=1}^{k} f_i(\cdot)$ which, due to being a separable sum, can be reduced to evaluating the proximal operator for the separate $f_i$ terms \cite{parikh2014proximal, chen2016direct}. The resulting algorithm is summarized in Algorithm \ref{algo:ADMM}.

\begin{algo}[H]
\centering
\begin{minipage}{.6\textwidth}
\begin{algorithmic}[1]
\STATE{ \textbf{Input:} $A, \vec{b}, k, \{(\prox_{f_i}^{\sigma}, L_i)\}_{i = 1}^{k}, \rho, \{(\vec{y}_i, \vec{u}_i)\}_{i = 1}^{k}$}

\WHILE{Convergence criterion not met}

\STATE{\begin{varwidth}[t]{\linewidth}Compute $\vec{x} \leftarrow \argmin_{\vec{z} \in \reals^n}\big\{
\frac{1}{2}\|A\vec{x} - \vec{b}\|_2^2$

\hspace{10em}$+ \frac{\rho}{2} \sum_{i=1}^{k} \|L_i\vec{x} - \vec{y}_i + \vec{u}_i\|_2^2
\big\}$\end{varwidth}} 
\STATE{Compute $\vec{y}_i \leftarrow \prox_{f_i}^{\sigma/\rho}(L_i\vec{x} + \vec{u}_i)$ for $i = 1,\dots,k$} 
\STATE{Compute $\vec{u}_i \leftarrow \vec{u}_i + L_i\vec{x} - \vec{y}_i$ for $i = 1,\dots,k$} 
\ENDWHILE

\RETURN{$\vec{x}$}

\end{algorithmic}
\end{minipage}
\caption{Alternating Direction Method of Multipliers for separable regularization terms.}
\label{algo:ADMM}
\end{algo}

For the efficient computation of the proximal operators $\prox_{f_i}^{\sigma}$, we have the following cases. If $f_i$ is the indicator function of a closed and convex set, then the proximal operator simplifies to the Euclidean projection onto that set. Such a projection is efficient to compute for various practical sets like componentwise bounds and balls. If $f_i(\vec{x}) = \|\vec{x}\|_1$, then the proximal operator corresponds to the soft thresholding operator \cite{bauschke2011convex}.

Total variation, i.e., $g(\vec{x}) = \|L_i\vec{x}\|_1$ with first order finite difference matrix $L_i \in \reals^{(n-1)\times n}$, can be modeled in different ways. If we consider $g(\vec{x}) = f_i(L_i\vec{x})$, then it fits with the $l_1$ case mentioned above. Alternatively, for a one-dimensional signal, the proximal operator $\prox_{\|L_i\cdot\|_1}^{\sigma}$ can be computed efficiently, see \cite{condat2013direct}.

\section{Miscellaneous lemmas and proofs}

\begin{lemma}\label{lemma:pushforward_support}
Let $\mu$ be a probability measure on $\reals^n$ and $g : D\subset \reals^n \rightarrow \reals^n$ be continuous, then
\begin{equation*}
    \supp(\mu \circ g^{-1}) = \cl(g(\supp(\mu))).
\end{equation*}
\end{lemma}

\begin{proof}
Let $\vec{y} = g(\vec{x}) \in g(\supp(\mu))$ and let $\set{N}_{\vec{y}}$ be any open neighborhood of $\vec{y}$, then 
\begin{align*}
    (\mu \circ g^{-1})(\set{N}_{\vec{y}}) = \mu(g^{-1}(\set{N}_{\vec{y}})) = \mu(\hat{\set{N}}_{\vec{x}} ) > 0,
\end{align*}
where $\hat{\set{N}}_{\vec{x}}$ is an open neighborhood of $\vec{x}$. By the support of $\mu \circ g^{-1}$ being closed, we get $\cl(g(\supp(\mu))) \subseteq \supp(\mu \circ g^{-1})$.

Let $\vec{y} \in \supp(\mu \circ g^{-1})$ and assume that $\vec{y} \not \in \cl(g(\supp(\mu)))$, then there exists an open neighborhood $\set{N}_{\vec{y}}$ around $\vec{y}$ such that $\set{N}_{\vec{y}} \cap g(\supp(\mu)) = \emptyset$. Hence, $g^{-1}(\set{N}_{\vec{y}}) \cap \supp(\mu) = \emptyset$. This would imply that $(\mu \circ g^{-1})(\set{N}_{\vec{y}}) = 0$, which contradicts $\vec{y} \in \supp(\mu \circ g^{-1})$. Therefore we can conclude that $\supp(\mu \circ g^{-1}) \subseteq \cl(g(\supp(\mu)))$.
\end{proof}

\begin{lemma}\label{lemma:locally_Lipschitz_null_sets}
Let $f : \reals^n \rightarrow \reals \cup \{\infty\}$ be a locally Lipschitz continuous function with $\dom(f)$ open. Let $\set{U} \subseteq \dom(f)$ with $\text{Leb}(\set{U}) = 0$, then $\text{Leb}(f(\set{U})) = 0$, i.e., locally Lipschitz continuous functions preserve sets of Lebesgue measure zero.
\end{lemma}
\begin{proof}
It follows from \cite[Lemma 7.25]{rudin1987real} that Lipschitz continuous functions preserve sets of Lebesgue measure zero. If $f$ is not Lipschitz continuous, consider the covering $\dom(f) = \cup_{i \in \mathbb{N}}\set{K}_i$, where $\set{K}_i$ is an increasing sequence of compact sets. Because 
\begin{equation*}
    \text{Leb}(f(\set{U})) \leq \sum_{i \in \mathbb{N}} \text{Leb}(f(\set{U} \cap \set{K}_i)), 
\end{equation*}
it is enough to show that $\text{Leb}(f(\set{U} \cap \set{K})) = 0$ for any compact set $\set{K}$.

Because $f$ is locally Lipschitz continuous, for every $\vec{u} \in \set{U} \cap \set{K}$ there exists an open ball $\set{B}(\vec{u}, \epsilon)$ on which $f$ is Lipschitz continuous. Furthermore, because the set of all these balls is an open cover of the compact set $\cl(\set{U} \cap \set{K}) \subset \dom(f)$, there exist a finite subcover $\{\set{B}(\vec{u}_i, \epsilon_i)\}_{i = 1}^{k}$. Therefore we can conclude that
\begin{equation*}
    \text{Leb}(f(\set{U} \cap \set{K})) \leq \sum_{i = 1}^{k} \text{Leb}(f(\set{U} \cap \set{K} \cap \set{B}(\vec{u}_i, \epsilon_i))) = 0. 
\end{equation*}
\end{proof}

\begin{corollary}\label{lemma:locally_Lipschitz_density}
Let $f : \reals^n \rightarrow \reals \cup \{\infty\}$ be a locally Lipschitz continuous function with $\dom(f)$ open and let $\mu$ be a Borel measure on $\dom(f)$. If $\mu \ll \text{Leb}$, then $\mu \circ f \ll \text{Leb}$, i.e., if $\mu$ has a density with respect to the Lebesgue measure, then so does $\mu \circ f$.
\end{corollary}

\begin{lemma}\label{lemma:sweeping}
Let $\set{U} \subseteq \reals^k$ and $\set{V} \subseteq \reals^l$ be compact sets with continuous functions $\phi:\set{U} \times \set{V} \rightarrow\reals^n$ and $\psi:\set{U}\times\set{V}\rightarrow\reals^n$. If $h : \set{U} \times \set{V} \rightarrow \phi(\set{U}, \set{V}) + \psi(\set{U}, \set{V})$ defined by $h(\vec{u}, \vec{v}) = \phi(\vec{u}, \vec{v}) + \psi(\vec{u}, \vec{v})$ is one-to-one, then $h$ is a homeomorphism between $\set{U} \times \set{V}$ and the Minkowski sum $\phi(\set{U}, \set{V}) + \psi(\set{U}, \set{V})$.
\end{lemma}
\begin{proof}
Note that $h$ is a bijective continuous function from the compact set $\set{U} \times \set{V}$ to the Hausdorff space $\phi(\set{U}, \set{V}) + \psi(\set{U}, \set{V}) \subset \reals^n$, hence by \cite[Proposition 5.2.5]{waldmann2014topology}, $h$ is a homeomorphism.
\end{proof}

\end{document}